\documentclass{iopart}
\usepackage{iopams,graphicx,latexsym,pdfsync, color}

\newcommand{\dk}{\, \diff k}
\newcommand{\dx}{\, \diff x}
\newcommand{\dy}{\, \diff y}

\newcommand{\Schwartz}{\mathcal{S}}
\newcommand{\R}{\mathbb{R}}
\newcommand{\Z}{\mathbb{Z}}
\newcommand{\N}{\mathbb{N}}

\newcommand{\I}{\mathrm{i}}
\newcommand{\bigO}{\mathcal{O}}
\newcommand{\diff}{\mathrm{d}}
\newcommand{\Diff}{\mathrm{D}}

\newcommand{\supp}{\mathop{\mathrm{supp}}}

\newcommand{\proof}{{\bf Proof }}

\newcommand{\qed}{\hfill$\Box$\smallskip}

\newcommand{\tfrac}[2]{{\textstyle\frac{#1}{#2}}}

\newcommand{\afl}{\fl\quad}

\newtheorem{theorem}{Theorem}[section]
\newtheorem{lemma}[theorem]{Lemma}
\newtheorem{corollary}[theorem]{Corollary}
\newtheorem{proposition}[theorem]{Proposition}

\newtheorem{remark}[theorem]{Remark}

\begin{document}

\title[Solitary waves for the FDKP equation]{Small-amplitude fully localised solitary waves for the full-dispersion Kadomtsev--Petviashvili equation}
\author{Mats Ehrnstr\"om}
\address{Department of Mathematical Sciences, Norwegian University of Science and Technology, 7491 Trondheim, Norway}
{ \ead{mats.ehrnstrom@ntnu.no}}
\author{Mark D. Groves}
\address{Fachrichtung 6.1 - Mathematik, Universit\"at des Saarlandes, Postfach 151150, 66041 Saarbr\"ucken, Germany}
\address{Department of Mathematical Sciences, Loughborough University, Loughborough, LE11 3TU, UK}
\ead{groves@math.uni-sb.de}

\begin{abstract}
The KP-I equation
\[
(u_t-2uu_x+\tfrac{1}{2}(\beta-\tfrac{1}{3})u_{xxx})_x -u_{yy}=0
\]
arises as a weakly nonlinear model equation for gravity-capillary waves with strong surface tension (Bond number
$\beta>1/3$). This equation admits --- as an explicit solution --- a `fully localised' or `lump' solitary wave which decays to zero in all spatial directions. Recently there has been interest in the \emph{ {full-dispersion} KP-I equation}
\[u_t + m(\Diff) u_x + 2 u u_x  = 0,\]
where $m(\Diff)$ is the Fourier multiplier with symbol
\[
m(k) = \left( 1 + \beta  |k|^2|\right)^{\frac{1}{2}} \left( \frac{\tanh  |k|}{|k|} \right)^{\frac{1}{2}} \left( 1 + \frac{2k_2^2}{k_1^2} \right)^{\frac{1}{2}},
\]
which is obtained by retaining the exact dispersion relation from the water-wave problem.
In this paper we show that the FDKP-I equation also has a fully localised solitary-wave solution.
The existence theory is variational and perturbative in nature. 
A variational principle for fully localised solitary waves
is reduced to a locally equivalent variational principle featuring
a perturbation of the variational functional associated with fully localised solitary-wave
solutions of the {KP-I} equation. A nontrivial critical point of the reduced functional is found
by minimising it over its natural constraint set.

\end{abstract}
\ams{35Q53, 35A15, 76B15}
\maketitle

\section{Introduction}

There has recently been considerable interest in `full dispersion' versions of model equations
obtained by modifying their dispersive terms so that their dispersion relation coincides with that of
the original physical problem. {The method has been used for some time in engineering and oceanography, but has become more attractive to mathematicians interested in nonlocal equations in view of improved use of harmonic analysis in partial differential equations. The prototypical example is the full-dispersion equation derived by Whitham \cite{Whitham67} as an alternative to the celebrated { Korteweg--de Vries} equation for water waves by incorporating the same linear dispersion relation as the full two-dimensional water-wave problem. It was shown by Ehrnstr\"{o}m, Groves \& Wahl\'{e}n \cite{EhrnstroemGrovesWahlen12} that the Whitham equation admits small-amplitude solitary-wave solutions which are approximated by scalings of the Korteweg--de Vries solitary wave; these waves are symmetric and of exponential decay rate (Bruell, Ehrnstr\"{o}m \& Pei \cite{BruellEhrnstroemPei17}). Other examples of current interest in fully dispersive equations include analytical investigations of bidirectional models in the spirit of Whitham (Ehrnstr\"{o}m, Johnson \& Claassen \cite{EhrnstroemJohnsonClaassen16}, Hur \& Tao \cite{HurTao16}) and Green-Naghdi (Duchene, Nilsson \& Wahl\'{e}n \cite{DucheneNilssonWahlen17}), as well as studies of the numerical, laboratory and modelling
properties of these equations (see respectively Claassen \& Johnson \cite{ClaassenJohnson17},
Carter \cite{Carter17} and Klein \emph{et al.} \cite{KleinLinaresPilodSaut17}). The monograph by Lannes \cite{Lannes} has a separate section on the subject of improved frequency dispersion. From a mathematical point of view, such equations often pose extra challenges arising from their more complicated symbols (which are typically inhomogeneous).

A higher-dimensional  example is given by the full-dispersion Kadomtsev--Petviashvili} (FDKP) equation
\begin{equation}\label{eq:FDKP}
u_t + m(\Diff) u_x + 2 u u_x  = 0,
\end{equation}
where the Fourier multiplier $m$ is given by
\[m(\Diff) = \left( 1 + \beta |\Diff|^2 \right)^{\frac{1}{2}} \left( \frac{\tanh  |\Diff|}{|\Diff|} \right)^{\frac{1}{2}} \left( 1 + \frac{ 2\Diff_2^2}{\Diff_1^2} \right)^{\frac{1}{2}}\]
with $\Diff = -\I(\partial_x, \partial_y)$, which was introduced by Lannes \cite{Lannes} (see also Lannes \&
Saut \cite{LannesSaut14}) as an alternative to the classical KP equation
\begin{equation}\label{eq:KP}
(u_t-2uu_x+\tfrac{1}{2}(\beta-\tfrac{1}{3})u_{xxx})_x -u_{yy}=0.
\end{equation}
Equation \eref{eq:KP} arises as a weakly nonlinear approximation for three-dimensional
gravity-capillary water waves, the parameter $\beta>0$
measuring the relative strength of surface tension; the cases $\beta>\frac{1}{3}$
(`strong surface tension') and $\beta<\frac{1}{3}$ (`weak surface tension') are termed
respectively KP-I and KP-II.

A \emph{(fully localised) FDKP solitary wave} is a nontrivial, evanescent solution of \eref{eq:FDKP} of the form $u(x,y,t)=u(x-ct,y)$
with wave speed $c>0$, that is, a homoclinic solution of the equation
\begin{equation}\label{eq:steady FDKP}
-c u + m(\Diff)u +  u^2 = 0.
\end{equation}
Similarly, a \emph{(fully localised) KP solitary wave} is a nontrivial, evanescent solution of \eref{eq:KP} of the form $u(x,y,t)=u(x-\tilde{c}t,y)$
with wave speed $\tilde{c}>0$, that is, a homoclinic solution of the equation
\begin{equation}\label{eq:steady KP}
(\tilde{c} -1)u  + \tilde{m}(\Diff)u +  u^2 = 0,
\end{equation}
where
\[\tilde{m}(\Diff) = 1+\frac{D_2^2}{D_1^2} - \tfrac{1}{2}(\beta- \tfrac{1}{3})D_1^2.\]

Note that the KP wave speed $\tilde{c}$ can be normalised to unity by the transformation
$u(x,y) \mapsto \tilde{c}u(\tilde{c}^\frac{1}{2}x,\tilde{c}y)$, which
converts \eref{eq:steady KP} into the equation
\begin{equation}
\tilde{m}(\Diff)u +  u^2 = 0. \label{eq:normalised steady KP}
\end{equation}
It is known that the KP-II equation does not admit any 
solitary waves (de Bouard \& Saut \cite{deBouardSaut97b}), while the explicit solutions
\begin{equation}
u(x,y)=-12\frac{3-X^2+Y^2}{(3+X^2+Y^2)^2}, \qquad
(X,Y)=\left(\tfrac{1}{2}(\beta-\tfrac{1}{3})\right)^{\!-\frac{1}{2}}\!\!(x,y)
 \label{eq:explicit floc}
\end{equation}
of \eref{eq:normalised steady KP} define KP-I  solitary waves
(see Figure \ref{fig:lump}). 
In this paper we demonstrate the existence of  solitary-wave solutions to the {FDKP-I} equation
and show how they are approximated by scalings of KP-I solitary waves.
(It is not known whether the latter are given by the explicit
formula \eref{eq:explicit floc}, but recent evidence points in this direction (see Chiron \& Scheid
\cite{ChironScheid17} and Liu \& Wei \cite{LiuWei17}).)

\begin{figure}
\centering
\includegraphics[width=4.5cm]{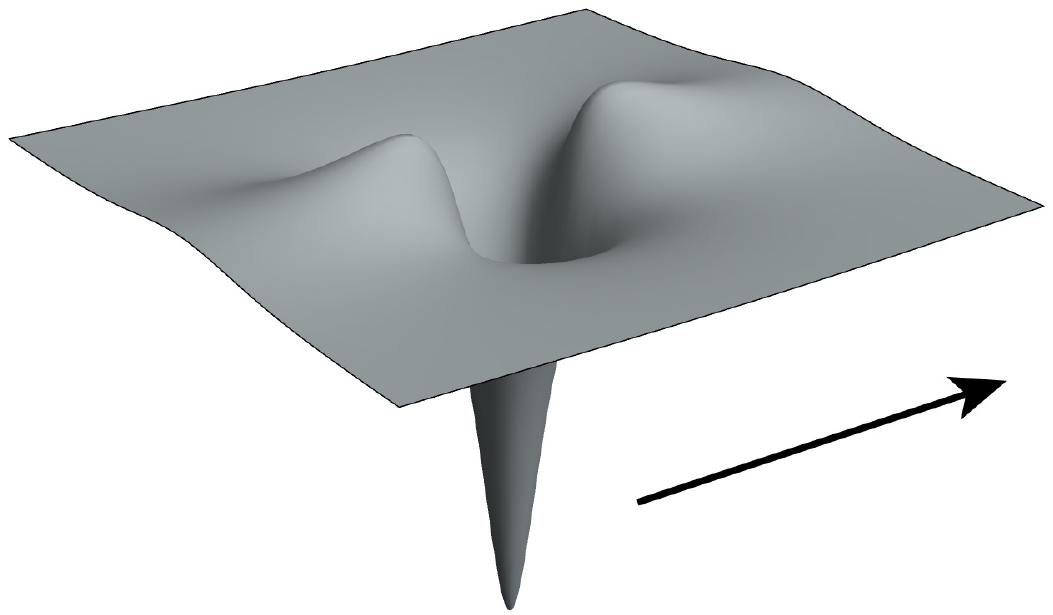}

\caption{Sketch of the KP-I  solitary wave given by \eref{eq:explicit floc}}
\label{fig:lump}
\end{figure}

\begin{theorem} \label{thm:main result 1}
There exists a  solitary-wave solution of the {FDKP-I} equation with speed $c=1-\varepsilon^2$
for each sufficiently small value of $\varepsilon>0$. This solution belongs to $H^\infty(\R^2)$ and
has polynomial decay rate.
\end{theorem}

An FDKP solitary wave is characterised as a critical point of the Hamiltonian
\begin{equation}
{\mathcal E}(u) = \frac{1}{2} \int_{\R^2} | (m(\Diff))^{\frac{1}{2}} u |^2 \dx \dy + \frac{1}{3} \int_{\R^2} u^3 \dx\dy \label{eq:EFDKP}
\end{equation}
subject to the constraint that the momentum
\begin{equation}\label{eq:I}
{\mathcal M}(u) = \frac{1}{2} \int_{\R^2} u^2 \dx \dy
\end{equation}
is fixed; the Lagrange multiplier is the wave speed $c$. Using this observation we
may reformulate the existence statement in Theorem \ref{thm:main result 1} in terms of
the calculus of variations. Let $X$ denote the
completion of $\partial_{x} \Schwartz(\R^2)$ with respect to the norm
\[
|u|_{X}^2 = \int_{\R^2} \left( 1 + \frac{k_2^2}{k_1^2} + \frac{k_2^4}{k_1^2} + |k|^{2s}\right) |\hat u(k)|^2\,\diff k,
\]
where $s>\frac{3}{2}$ and $\Schwartz(\R^2)$ is the two-dimensional Schwartz space.

\begin{theorem} \label{thm:main result 2}
Suppose that $\beta>\frac{1}{3}$. The formula
${\mathcal I}_\varepsilon={\mathcal E}-c{\mathcal M}$ with $c=1-\varepsilon^2$ defines a smooth
functional ${\mathcal I}_\varepsilon: X \to \R$
which has a nontrivial critical point for each sufficiently small value of $\varepsilon>0$.
\end{theorem}

To motivate our main result it is instructive to review the formal derivation of the (normalised) steady KP equation
\eref{eq:normalised steady KP} from the steady FDKP equation \eref{eq:steady FDKP}.
We begin with the linear dispersion relation for a two-dimensional sinusoidal
travelling wave train with wave number $k_1$ and speed $c$, namely
\[c= \left( 1 + \beta |k_1|^2 \right)^{\frac{1}{2}} \left( \frac{\tanh  |k_1|}{|k_1|} \right)^{\frac{1}{2}}\]
The function $k_1 \mapsto c(k_1)$, $k_1 \geq 0$ has a unique global minimum at $k_1=0$ with $c(0)=1$ (see Figure 
\ref{fig:dispersion relation}). Bifurcations of nonlinear solitary waves {are} expected whenever the
linear group and phase speeds are equal, so that $c^\prime(k_1)=0$ (see Dias \& Kharif \cite[\S 3]{DiasKharif99});
one therefore expects bifurcation of small-amplitude solitary waves from uniform flow with unit speed. 
Furthermore, observing
that $m$ is an analytic function of $k_1$ and $\frac{k_2}{k_1}$ (note that
$|k|^2 = k_1^2 + \frac{k_2^2}{k_1^2}k_1^2$), one finds that
\begin{equation}
m(k)=\tilde{m}(k) + O(|(k_1,\tfrac{k_2}{k_1})|^4) \label{eq:FDKP to KP}
\end{equation}
as $(k_1,\frac{k_2}{k_1}) \to 0$.
We therefore make the steady-wave \emph{Ansatz} $u(x,y,t)=\tilde{u}(x-ct,y)$ and substitute $c=1-\varepsilon^2$
and
\begin{equation}
 \tilde{u}(x,y) = \varepsilon^2\zeta(\varepsilon x, \varepsilon^2 y)
\label{eq:KP scaling}
\end{equation}
into equation \eref{eq:steady FDKP}. This calculation shows that to leading order $\zeta$ satisfies
\begin{equation}
\tilde{m}(D)\zeta +  \zeta^2 =0, \label{eq:normalised steady KP again}
\end{equation}
which is the Euler--Lagrange equation for the (smooth) functional
${\mathcal T}_0: \tilde{Y} \to \R$ given by
\[{\mathcal T}_0(\zeta)=  \frac{1}{2} \int_{\R^2} | (\tilde{m}(\Diff))^{\frac{1}{2}} \zeta |^2 \dx \dy + \frac{1}{3} \int_{\R^2} \zeta^3 \dx\dy,\]
where $\tilde{Y}$ is the completion of $\partial_x \Schwartz(\R^2)$ with respect to the norm
\[|\zeta|_{\tilde Y}^2= \int_{\R^2} \left(1+\tfrac{k_2^2}{k_1^2}+k_1^2\right)|\hat{\zeta}|^2\,\dk.\]
We proceed by performing a rigorous local variational reduction
which converts ${\mathcal I}_\varepsilon$ to a perturbation ${\mathcal T}_\varepsilon$ of ${\mathcal T}_0$
(Section \ref{sec:reduction}).

\begin{figure}
\centering
\includegraphics[width=4.5cm]{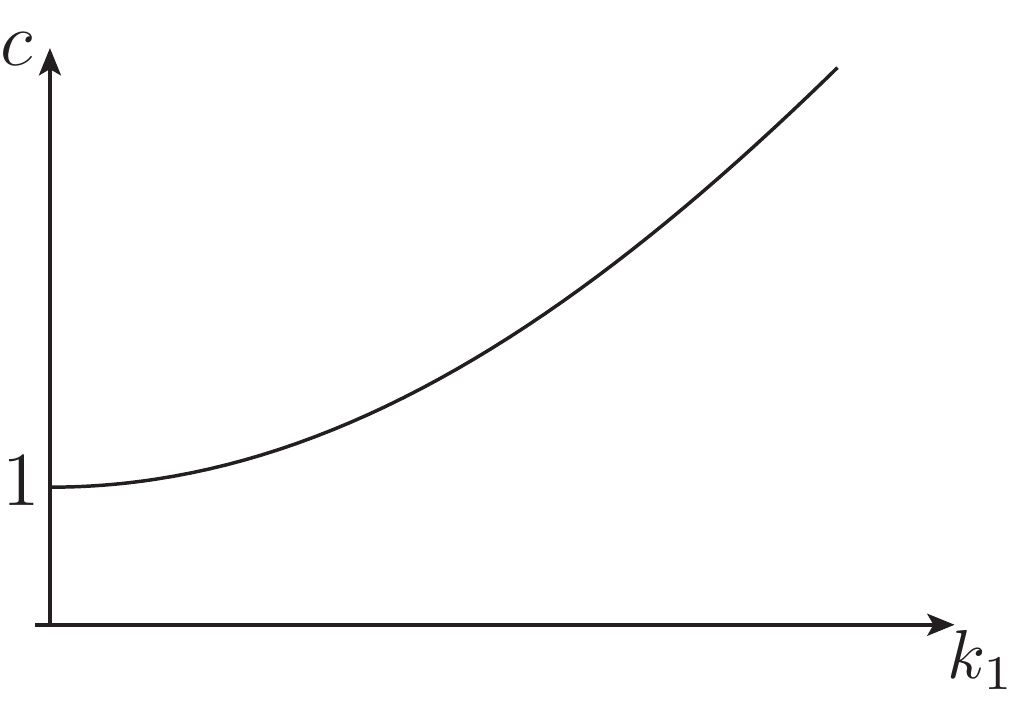}

\caption{FKDP-I dispersion relation for two-dimensional wave trains}
\label{fig:dispersion relation}
\end{figure}

\begin{figure}
\centering
\includegraphics[scale=0.67]{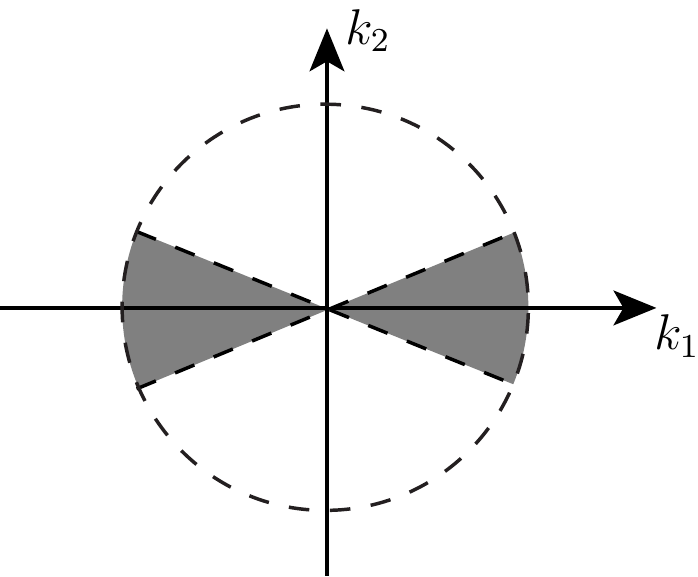}

\caption{The cone $C = \{ k \in \R^2 \colon |k| \leq \delta, \tfrac{|k_2|}{|k_1|} \leq \delta \}$ cut out of the closed ball $\{k \in \R^2 \colon |k| \leq \delta\}$ in $\R^2$.}
\label{fig:bow tie}
\end{figure}

The estimate \eref{eq:FDKP to KP} suggests that the spectrum of
a  solitary wave $ u(x,y)$ is concentrated in the region $|k_1|, |\frac{k_2}{k_1}| \ll 1$.
We therefore
decompose $u$ into the sum
of functions $u_1$ and $u_2$ whose spectra are supported in the
region
\[  C=\left\{(k_1,k_2) \colon |k| \leq \delta, \left|\tfrac{k_2}{k_1}\right| \leq \delta\right\}\]
and its complement, where $\delta$ is a small positive number (see Figure \ref{fig:bow tie}), so that
\[u_1 = \chi(\Diff)u, \qquad u_2=(1-\chi(\Diff))u,\]
in which
$\chi$ is the characteristic function of $C$. In Section \ref{sec:reduction} we employ a method
akin to the variational Lyapunov-Schmidt reduction
to determine $u_2$ as a function of $u_1$ and thus obtain the reduced functional
${\mathcal J}_\varepsilon: U \to \R$ given by
\[{\mathcal J}_\varepsilon(u_1)={\mathcal I}_\varepsilon(u_1+u_2(u_1));\]
here $U = \{u_1 \in X_1: |u_1|_\varepsilon \leq 1\}$ is the unit ball
in the space $(X_1,|\cdot|_\varepsilon)$,
in which $X_1=\chi(D)X$ and $|\cdot|_\varepsilon$ is the scaled norm
\[
|u_1|_\varepsilon^2 = \int_{\R^2} \left( 1 +  \varepsilon^{-2} \frac{k_2^2}{k_1^2} + \varepsilon^{-2} k_1^2\right) |\hat u_1(k)|^2\, \diff k.
\]
Applying the KP scaling \eref{eq:KP scaling} to $u_1$, one finds that
${\mathcal J}_\varepsilon(u_1)=\varepsilon^3 {\mathcal T}_\varepsilon(\zeta)$, where
\[
{\mathcal T}_\varepsilon(\zeta)={\mathcal T}_0(\zeta)+\varepsilon^{\frac{1}{2}}{\mathcal R}_\varepsilon(\zeta), \qquad {\mathcal R}_\varepsilon(\zeta) \lesssim |\zeta|_{\tilde Y}^2\]
(with corresponding estimates for the derivatives of the remainder term).
Each critical point $\zeta_\infty$ of ${\mathcal T}_\varepsilon$ with $\varepsilon>0$
corresponds to a critical point $u_{1,\infty}$ of
$\mathcal J_\varepsilon$, which in turn defines a critical point $u_{1,\infty}+u_2(u_{1,\infty})$ of
${\mathcal I}_\varepsilon$.

We study ${\mathcal T}_\varepsilon$
in a fixed ball
\[
B_M(0)=\{\zeta \colon |\zeta|_{\tilde Y}<M\}, 
\]
in the space
$(\tilde{Y}_\varepsilon,|\cdot|_{\tilde Y})$, where $\tilde{Y}_\varepsilon=\chi_\varepsilon(D)\tilde{Y}$ and
$\chi_\varepsilon(k_1,k_2)=\chi(\varepsilon k_1, \varepsilon^2 k_2)$.  The parameters $M$ and $\varepsilon$ are related in the following manner: for any $M>1$ there exists
$\varepsilon_M \lesssim M^{-2}$ such that all estimates hold uniformly over $\varepsilon \in
[0,\varepsilon_M]$. We do not make make these threshold values of $\varepsilon$ explicit; it is simply assumed that $\varepsilon_M$ is taken sufficiently small.
In the limit $\varepsilon=0$
we can set $M=\infty$ and recover the KP variational functional ${\mathcal T}_0: \tilde{Y} \to \R$
(note that $\tilde{Y}_0=\tilde{Y}$).  In fact ${\mathcal T}_\varepsilon: B_M(0) \to \R$ may be considered as a perturbation
of the `limiting' functional ${\mathcal T}_0: \tilde{Y} \to \R$. More precisely
$\varepsilon^\frac{1}{2}{\mathcal R}_\varepsilon \circ\chi_\varepsilon(D)$ (which coincides with
$\varepsilon^\frac{1}{2}{\mathcal R}_\varepsilon$ on $B_M(0) \subset \tilde{Y}_\varepsilon$) converges uniformly
to zero over $B_M(0) \subset \tilde{Y}$ (with corresponding uniform convergence for its derivatives),
and we study ${\mathcal T}_\varepsilon$ by perturbative arguments in this spirit.

In Section \ref{sec:existence} we seek critical points of ${\mathcal T}_\varepsilon$ by minimising it on its
\emph{natural constraint set}
\[N_\varepsilon=\{\zeta\in B_M(0):\zeta\neq 0,\mathrm{d}{\mathcal T}_\varepsilon[\zeta](\zeta)=0\},\]
our motivation being the observation that the critical points of ${\mathcal T}_\varepsilon$ coincide with
those of ${\mathcal T}_\varepsilon|_{N_\varepsilon}$.
The natural constraint set has a geometrical interpretation (see Figure \ref{fig:ncs geometry}),
namely that any ray in $B_M(0)$ intersects the natural constraint manifold
$N_\varepsilon$ in at most one point and the value of ${\mathcal T}_\varepsilon$ along
such a ray attains a strict maximum at this point.
This fact is readily established by a direct calculation for $\varepsilon=0$ and deduced by a
perturbation argument for $\varepsilon>0$, and similar perturbative methods yield the existence of a 
a sequence $\{\zeta_n\} \subset B_{M-1}(0)$
with
\[{\mathcal T}_\varepsilon|_{N_\varepsilon} \to \inf {\mathcal T}_\varepsilon|_{N_\varepsilon}>0, \qquad
|\mathrm{d}{\mathcal T}_\varepsilon[\zeta_n]|_{\tilde{Y}_\varepsilon \to \R} \to 0\]
as $n \to \infty$.
The following theorem is established by applying weak continuity arguments to minimising sequences
of the above kind.

\begin{theorem} \label{thm:variational 1}
Let $\{\zeta_n\} \subset B_{M-1}(0)$ be a minimising sequence for ${\mathcal T}_\varepsilon|_{N_\varepsilon}$
with
$|\mathrm{d}{\mathcal T}_\varepsilon[\zeta_n]|_{\tilde{Y}_\varepsilon \to \R}\to 0$ as $n \to \infty$.
There exists $\{w_n\} \subset \Z^2$ with the property that a subsequence of
$\{\zeta_n(\cdot+w_n)\}$
converges weakly in $\tilde{Y}_\varepsilon$ to a nontrivial critical point $\zeta_\infty$ of
${\mathcal T}_\varepsilon$.
\end{theorem}

The short proof of Theorem \ref{thm:variational 1} does not show that the critical point $\zeta_\infty$ is a \emph{ground state},
{ that is,} a minimiser of ${\mathcal T}_\varepsilon$ over $N_\varepsilon$.
This deficiency is removed in Section \ref{sec:ground states} with the help of an abstract version of the
concentration-compactness principle due to Buffoni, Groves \& Wahl\'{e}n \cite[Appendix A]{BuffoniGrovesWahlen18}. (That paper treats fully localised solitary waves in the Euler equations with weak surface tension using theory closely connected to ours.) 

\begin{theorem} \label{thm:variational 2}
Let $\{\zeta_n\} \subset B_{M-1}(0)$ be a minimising sequence for ${\mathcal T}_\varepsilon|_{N_\varepsilon}$
with\linebreak $|\mathrm{d}{\mathcal T}_\varepsilon[\zeta_n]|_{\tilde{Y}_\varepsilon \to \R}\to 0$ as $n \to \infty$.
There exists $\{w_n\} \subset \Z^2$ with the property that a subsequence of
$\{\zeta_n(\cdot+w_n)\}$
converges weakly, and strongly if $\varepsilon=0$, in $\tilde{Y}_\varepsilon$ to a ground state $\zeta_\infty$.
\end{theorem}

\begin{figure}
\centering
$\mbox{}$\hspace{3cm}
\includegraphics[width=3cm]{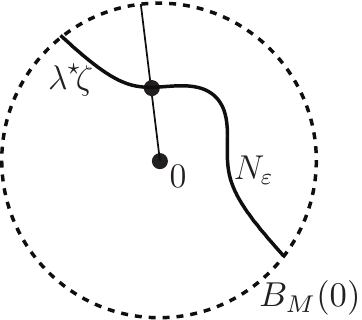}
\hspace{0.6cm}
\includegraphics[width=4cm]{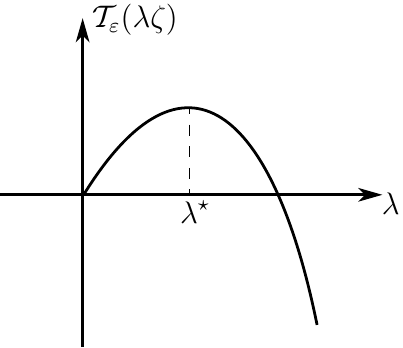}

\caption{Any ray intersects the natural constraint set $N_\varepsilon$
in at most one point and the value of ${\mathcal T}_\varepsilon$ along such a ray attains a strict maximum at this point}
\label{fig:ncs geometry}
\end{figure}

We prove Theorems \ref{thm:variational 1} and \ref{thm:variational 2} for $\varepsilon=0$
and $\varepsilon>0$ separately, in the latter case taking advantage of the relationship ${\mathcal I}_\varepsilon(u)
=\varepsilon^3{\mathcal T}_\varepsilon(\zeta)$, where $u=u_1(\zeta)+u_2(u_1(\zeta))$,
and the fact that
$\tilde{Y}_\varepsilon$ coincides with { $H^s_\varepsilon(\R^2):=\chi_\varepsilon(D)H^s(\R^2)$ for any $s > \frac{3}{2}$}.
The function $u_\infty = u_1(\zeta_\infty) + u_2(u_1(\zeta_\infty))$ given by these theorems
is then a nontrivial critical point of ${\mathcal I}_\varepsilon$, which concludes the proof of Theorem \ref{thm:main result 2}.
The discussion of the case $\varepsilon=0$ does not contribute to this existence proof
but shows that the KP ground states (that is, the ground states of ${\mathcal T}_0$) are characterised in the same
way as the ground states of ${\mathcal T}_\varepsilon$ for $\varepsilon>0$. Using this information, we show that
the ground states of ${\mathcal T}_\varepsilon$ converge to those of ${\mathcal T}_0$ as $\varepsilon \to 0$
in the following sense.

\begin{theorem} \label{thm:convergence of ground states}
Let $c_\varepsilon = \inf_{N_\varepsilon} {\mathcal T}_\varepsilon$.
\begin{itemize}
\item[(i)]
One has that $\lim_{\varepsilon \to 0} c_\varepsilon = c_0$.
\item[(ii)]
Let $\{\varepsilon_n\}$ be a sequence with $\lim_{n \to \infty} \varepsilon_n=0$ and
$\zeta^{\varepsilon_n}$ be a ground state of ${\mathcal T}_{\varepsilon_n}$.
There exists $\{w_n\} \subset \Z^2$ such that
a subsequence of $\{\zeta^{\varepsilon_n}(\cdot+w_n)\}$ converges strongly in $\tilde{Y}$
to a ground state $\zeta^\star$ of ${\mathcal T}_0$.
\end{itemize}
\end{theorem}

Our final result concerns convergence of FDKP-I solitary waves to KP-I solitary waves
and is obtained as a corollary of Theorem \ref{thm:convergence of ground states}.

\begin{theorem}
Let $\{\varepsilon_n\}$ be a sequence with $\lim_{n \to \infty} \varepsilon_n=0$ and
$u^{\varepsilon_n}$ be a critical point of ${\mathcal I}_{\varepsilon_n}$
with ${\mathcal I}_{\varepsilon_n}(u^{\varepsilon_n})=\varepsilon_n^3c_{\varepsilon_n}$,
so that the formula $u^{\varepsilon_n}=u_1(\zeta^{\varepsilon_n}) + u_2(u_1(\zeta^{\varepsilon_n}))$
defines a ground state $\zeta^{\varepsilon_n}$ of ${\mathcal T}_\varepsilon$. There exists
$\{w_n\} \subset \Z^2$ such that a subsequence of
$\{\zeta^{\varepsilon_n}(\cdot+w_n)\}$ converges strongly in $\tilde{Y}$
to a ground state $\zeta^\star$ of ${\mathcal T}_0$.
\end{theorem}

Defining
$u^\star_{\varepsilon}(x,y) = \varepsilon^2 \zeta^\star(\varepsilon x,\varepsilon^2 y)$,
so that $u^\star_\varepsilon$ is a KP solitary wave with wave speed $\varepsilon^2$,
we find that the difference $u^{\varepsilon_n} - u_{\varepsilon_n}^\star$ converges to zero
in $(\tilde{Y},|\cdot|_{\varepsilon_n})$ and in $H^{\frac{1}{2}}(\R^2)$ (see 
Remark~\ref{remark:u-convergence}). Although these functions
are small, their difference converges to zero faster than the functions themselves), so that we also
have convergence with respect to the original variables.

\begin{remark}
The results presented in this paper apply with straightforward modifications to the generalised FDKP-I
and KP-I equations obtained by replacing the nonlinear term $(u^2)_x$ by $(u^p)_x$ with
$2 \leq p < 5$ (see Proposition~\ref{prop:anisotropic embeddings}). The proof of the counterpart to Theorem \ref{thm:variational 1}
with $\varepsilon=0$ also yields a concise variational existence theory for gKP-I solitary waves as
an alternative to those already available in the literature (de Bouard \& Saut \cite{deBouardSaut97b},
Pankov \& Pfl\"{u}ger \cite{PankovPflueger99,PankovPflueger00}, Willem \cite[Ch 7]{Willem},
Wang, Ablowitz \& Segur \cite{WangAblowitzSegur94} and
Liu \& Wang \cite{LiuWang97}).
\end{remark}

\section{Function spaces}

In this section we introduce the function spaces (and basic properties thereof) which are used in the
variational reduction and existence theory in Sections \ref{sec:reduction} and \ref{sec:existence} below. For notational simplicity we generally omit the exact value of $\tfrac{1}{2}(\beta-\tfrac{1}{3})$ and treat it as being of unit size (without this simplification the term $k_1^2$ in the norm for
$\tilde{Y}$ is multiplied by $\tfrac{1}{2}(\beta-\tfrac{1}{3})$, which does not affect the
proof in any way.) Examining the quadratic parts of the variational functionals
\begin{eqnarray*}
\mathcal{I}(u) & = & {\mathcal E}(u) - c{\mathcal M}(u) \\
& = &  \frac{1}{2}\int_{\R^2} \left( (m(\Diff)^{\frac{1}{2}} u)^2 - c u^2 \right) \dx \dy + \frac{1}{3} \int_{\R^2} u^3 \dx \dy
\end{eqnarray*}
and
\[{\mathcal T}_0(\zeta)=  \frac{1}{2} \int_{\R^2} | (\tilde{m}(\Diff))^{\frac{1}{2}} \zeta |^2 \dx \dy + \frac{1}{3} \int_{\R^2} \zeta^3 \dx\dy\]
for the steady FDKP-I and KP-I equations \eref{eq:steady FDKP} and \eref{eq:normalised steady KP again}
shows that their natural energy spaces are the completions $Y$ and $\tilde{Y}$ of 
\[
\partial_{x} \Schwartz(\R^2) = \{ \partial_x f \colon f \in \Schwartz(\R^2)\},
\]
where $\Schwartz(\R^2)$ is the Schwartz space of rapidly decaying smooth functions,
with respect to the norms
\begin{equation}\label{eq:tildeY}
|u|_{Y}^2 =  \int_{\R^2} \left( 1 + \left|\frac{k_2}{k_1}\right| + \frac{|k|^{\frac{3}{2}}}{|k_1|^{\:\:\,}} \right) | \hat u(k)|^2 \, \diff k,
\end{equation}
and
\begin{equation}\label{eq:Y}
|u|_{\tilde Y}^2 = \int_{\R^2} \left( 1 +  \frac{k_2^2}{k_1^2} + k_1^2 \right) | \hat u(k)|^2 \, \diff k.
\end{equation}
Here ${\mathcal F}: u \mapsto \hat{u}$ denotes the unitary Fourier transform on $\Schwartz(\R^2)$. (In defining $|\cdot|_Y$ we have used the
fact that
\begin{eqnarray*}
& & \parbox{5cm}{$\displaystyle m(k) \simeq 1+\frac{|k_2|}{|k_1|}$,} |k| \leq \delta, \qquad \nonumber \\
& & \parbox{5cm}{$\displaystyle m(k) \simeq  |k_1|^\frac{1}{2}+\frac{|k_2|^\frac{3}{2}}{|k_1|}$,} |k| \geq \delta,
\end{eqnarray*}
for { any} $\delta>0$.)
Although the largest space continuously embedded into both $Y$ and $\tilde Y$ is defined by the weight $1 + k_2^2 k_1^{-2} + |k|$, we work in the smaller space $X$ defined as the completion of $\partial_{x} \Schwartz(\R^2)$ with respect to the norm
\begin{equation}\label{eq:X}
|u|_{X}^2 = \int_{\R^2} \left( 1 + \frac{k_2^2}{k_1^2} + \frac{k_2^4}{k_1^2} + |k|^{2s}\right) |\hat u(k)|^2\,\diff k,
\end{equation}
where the Sobolev index $s > \frac{3}{2}$ is fixed. Finally, we introduce the completion $Z$ of $\partial_{x} \Schwartz(\R^2)$ with respect to the norm
\begin{equation}\label{eq:Z}
|u|_{Z}^2 = \int_{\R^2} \left( 1 + |k| + k_1^2|k|^{2s-3}\right) |\hat u(k)|^2\,\diff k;
\end{equation}
it follows from Lemma \ref{lemma:global embeddings} and Remark \ref{rem:m iso} that ${ Z = m(D)X}$.


\begin{lemma}\label{lemma:global embeddings} \hspace{2cm}
\begin{itemize}
\item[(i)]
One has the continuous embeddings
\[\afl
X \hookrightarrow \tilde Y \hookrightarrow Y \hookrightarrow L^2(\R^2), \quad H^{s-\frac{1}{2}}(\R^2) \hookrightarrow Z \hookrightarrow L^2(\R^2), \quad X \hookrightarrow H^s(\R^2),
\]
and in particular $X \hookrightarrow \mathrm{BC}(\R^2)$, the space of bounded, continuous functions on $\R^2$.
\item[(ii)]
The Fourier multiplier $m(\Diff)$ maps $X$ continuously into $Z$.
\end{itemize}
\end{lemma}
\proof
(i) The first and second chain of embeddings follow from the estimates
\begin{eqnarray*}
1 &\leq & 1 + \left|\frac{k_2}{k_1}\right| + \frac{|k_2|^{\frac{3}{2}}}{|k_1|} + |k_1|^{\frac{1}{2}}\nonumber \\
& \lesssim & 1 + \frac{k_2^2}{k_1^2} + k_1^2 + \frac{|k_2|^{\frac{3}{2}}}{|k_1|}\nonumber \\
& \lesssim & 1 + \frac{k_2^2}{k_1^2} + k_1^2\nonumber \\
& \leq 1 & + \frac{k_2^2}{k_1^2} + \frac{k_2^4}{k_1^2} + |k|^{2s}
\end{eqnarray*}
(in the third step we multiply and divide the last term by $|k_1|^{\frac{1}{2}}$ and apply Young's inequality with $\frac{1}{4} + \frac{3}{4} = 1$), and
\[
1 \leq 1 + |k| + k_1^2 |k|^{2s-3} \lesssim 1 + |k|^{2s-1},
\]
while the third follows from the estimate
\[1+|k|^{2s}
\leq 1 + \frac{k_2^2}{k_1^2} + \frac{k_2^4}{k_1^2} + |k|^{2s}.\]
The embedding of $X$ into $\mathrm{BC}(\R^2)$ follows from $H^s(\R^2) \hookrightarrow \mathrm{BC}(\R^2)$ (because
$s > \frac{3}{2}$).

(ii) Observe that
\[(1+|k|+k_1^2|k|^{2s-3})m(k)^2 \lesssim 1+\frac{k_2^2}{k_1^2}\]
for $|k| \leq \delta$ and
\begin{eqnarray*}
(1+|k|+k_1^2|k|^{2s-3})m(k)^2 & \lesssim & (1+|k|+k_1^2|k|^{2s-3})|k|\left(1+\frac{k_2^2}{k_1^2}\right) \nonumber \\
& = & \frac{|k|^4}{k_1^2} + |k|^{2s} \nonumber \\
& \lesssim & \frac{k_2^4}{k_1^2}+k_1^2+|k|^{2s} \nonumber \\
& \lesssim & \frac{k_2^4}{k_1^2}+|k|^{2s}
\end{eqnarray*}
for $|k| \geq \delta$ (because $(1+\beta |k|^2)|k|^{-1}\tanh|k| \gtrsim |k|$ for $|k| \geq \delta$), so that\linebreak
$|m(\Diff)(\cdot)|_Z^2 \lesssim |\cdot|_X^2$.
\qed

The space $\tilde{Y}$ admits a local representation: the map
$w \mapsto u:=w_x$ is an isometric isomorphism $A \rightarrow \tilde{Y}$,
where $A$ is the completion of $\partial_x \Schwartz(\R^2)$ with respect to the norm
\[
| w |_A^2 =
\int_{\R^2} \left( w_x^2+w_y^2 + w_{xx}^2 \right) \, \diff x \, \diff y.
\]
In this spirit we can also define the localised space $A(Q_j)$, where
\[
Q_j = \{(x,y) \in \R^2: |x-j_1| < \tfrac{1}{2}, |y-j_2| < \tfrac{1}{2}\}
\] 
is the unit cube centered at the point $j = (j_1,j_2) \in \Z^2$, as the completion
of $\partial_x C^\infty(\overline{Q}_j)$ with respect to the norm
\[
| w |_{A(Q_j)}^2 = \int_{Q_j} (w_x^2+w_y^2 + w_{xx}^2 ) \, \diff x \, \diff y,
\]
and $\tilde{Y}(Q_j) = \partial_x A(Q_j)$ with 
$| u |_{\tilde Y(Q_j)}=|w|_{A(Q_j)}$.
Note that $u|_{Q_j}$ belongs to $\tilde Y(Q_j)$ for each $u \in \tilde{Y}$ and
\[|u|_{\tilde Y}^2 = \sum_{j \in \Z^2} |u|_{\tilde Y(Q_j)}^2.\]

\begin{proposition}\label{prop:anisotropic embeddings}
The space $\tilde{Y}$ is
\begin{itemize}
\item[(i)]
continuously embedded in $L^p(\R^2)$ for $2 \leq p \leq 6$,
\item[(ii)]
compactly embedded in $L^p_\mathrm{loc}(\R^2)$ for $2 \leq p < 6$.
\end{itemize}
Furthermore, the space $\tilde{Y}(Q_j)$ is continuously embedded in $L^p(Q_j)$
for $2 \leq p \leq 6$.
\end{proposition}
\begin{proof}
Part (i) and the assertion concerning $\tilde{Y}(Q_j)$ follow from
Besov, Ilin \& Nikolskii \cite[Thm 15.7]{BesovIlinNikolskii}
(applied to the local representations).
Part (ii) is an interpolation result between $L^2_\mathrm{loc}(\R^2)$
and $L^6_\mathrm{loc}(\R^2)$; de Bouard \& Saut
\cite[Lemma 3.3]{deBouardSaut97b} show that the inclusion 
$\tilde{Y} \subset L^2_\mathrm{loc}(\R^2)$ is compact, and the inclusion 
$\tilde{Y} \subset L^6_\mathrm{loc}(\R^2)$ is continuous by (i).\qed
\end{proof}

Our next results concern the functional ${\mathcal I}$ and its { Euler--Lagrange} equation.

\begin{corollary} \label{cor:spaces corollary} \hspace{2cm}
\begin{itemize}
\item[(i)]
The formula $u \mapsto -cu+m(\Diff)u { +} u^2$ maps $X$ smoothly into $Z$.
\item[(ii)]
The functional ${\mathcal I}$ maps $X$ smoothly into $\R$  and its critical points are precisely the
homoclinic solutions (in $X$) of equation \eref{eq:steady FDKP}.
\end{itemize}
\end{corollary}
\begin{proposition}
The functional ${\mathcal I}$ has no critical points that belong to
$L^1(\R^2) \cap L^\infty(\R^2)$. In particular, all bounded homoclinic solutions
of \eref{eq:steady FDKP} in $X$ have polynomial decay rate.
\end{proposition}
\proof
Suppose that $u \in L^1(\R^2) \cap L^\infty(\R^2)$ is a critical point of ${\mathcal I}$,
so that $u^2$, $m(D)u \in L^2(\R^2)$ and
\begin{equation}
m(k)\hat{u}(k) = c\hat{u}(k) - \widehat{u^2}(k). \label{eq:decay rate contra}
\end{equation}
Furthermore $\hat u$ and $\widehat{u^2}$ are both continuous (since
$u, u^2 \in L^1(\R^2)$), so that $m(k)\hat{u}(k)$ is also continuous. Recall that
\[m(k) \simeq 1+\frac{|k_2|}{|k_1|}, \qquad |k| \leq \delta,\]
so that $m(k)\hat{u}(k)$ is unbounded along a sequence $\{k_j\}$ fulfilling $k_{1,j} = k_{2,j}^2 \to 0$
as $j \to \infty$. This observation contradicts \eref{eq:decay rate contra}.
\qed
\begin{proposition} \label{prop:FDKP op weak}
The formula $u \mapsto -cu+m(\Diff)u  { +} u^2$  { defines} a weakly
continuous mapping $X \to Z$.
\end{proposition}
\proof Suppose that
$\{u_n\}$ converges weakly to $u$ in $X$ and hence weakly in $H^s(\R^2)$
and strongly in $L^4_\mathrm{loc}(\R^2)$.
It follows that $\langle u_n^2, \phi \rangle_{L^2}$ converges to 
$\langle u^2,\phi \rangle_{L^2}$ for each $\phi \in C_0^\infty(\R^2)$, so that
$\{u_n^2\}$ converges weakly to $u^2$ in $H^k(\R^2)$ for each integer $k \leq s$
and hence weakly in $Z$.\qed

We decompose $u \in L^2(\R^2)$ into the sum
of functions $u_1$ and $u_2$ whose spectra are supported in the region
\begin{equation}\label{eq:C}
C = \left\{ k \in \R^2 \colon |k| \leq \delta, \tfrac{|k_2|}{|k_1|} \leq \delta\right\}
\end{equation}
and its complement (see Figure \ref{fig:bow tie}) by writing
\[u_1 = \chi(\Diff)u, \qquad u_2=(1-\chi(\Diff))u,\]
where
$\chi$ is the characteristic function of $C$. Since $X$ is a subspace of $L^2(\R^2)$, the Fourier multiplier $\chi(\Diff)$ induces an orthogonal decomposition
\[
X = X_1 \oplus X_2,
\]
where
\[
X_1{ = } \chi(\Diff) X, \qquad  X_2 { = } (1 - \chi(\Diff)) X,
 \]
with analogous decompositions for the spaces $Y$, $\tilde Y$ and $Z$; we henceforth use the 
subscripts $_1$ and $_2$ to denote the corresponding orthogonal projections.

\begin{lemma}
The spaces $X_1$, $Y_1$, $\tilde{Y}_1$ and $Z_1$ all coincide with $\chi(\Diff)L^2(\R^2)$, and the
norms $|\cdot|_{L^2}$, $|\cdot|_{X}$, $|\cdot|_{Y}$, $|\cdot|_{\tilde Y}$ and $|\cdot|_{Z}$ are all equivalent
norms for these spaces.
\end{lemma}
\proof
Observe that
\begin{eqnarray*}
\chi(\Diff)L^2(\R^2) & = & \{u \in L^2(\R^2): \supp \hat{u} \subseteq C\}, \nonumber \\
\hspace{7.5mm}\chi(\Diff)X & = & \{u \in X: \supp \hat{u} \subseteq C\},
\end{eqnarray*}
so that
$X \subseteq L^2(\R^2)$ implies that $\chi(\Diff)X \subseteq \chi(\Diff)L^2(\R^2)$.
Conversely, suppose that $u \in L^2(\R^2)$ with $\supp \hat{u} \subseteq C$, so that
$|u|_X^2 \leq (1+2\delta^2) |u|_{L^2}^2$ and hence $u \in X$;
it follows that $\chi(\Diff)L^2(\R^2) \subseteq \chi(\Diff)X$. The other equalities are established in the
same way.\qed

Let us now consider the Fourier multipliers
\begin{equation}\label{eq:n defs}
n = m -1, \qquad \tilde n = \tilde m - 1
\end{equation}
which arise in our study of solitary waves with near unit speed.

\begin{lemma} \label{lemma:n}
The mapping $n(\Diff)$ is an isomorphism $X_2 \to Z_2$.
\end{lemma}
\proof
It follows from Lemma~\ref{lemma:global embeddings}(ii) that $n(\Diff) = m(\Diff) -1$ maps $X$ continuously
into $Z$ and hence $X_2$ continuously into $Z_2$.

Writing
\[
n(k)=\bigg(\bigg(1+\beta|k|^2)\frac{\tanh |k|}{|k|}\bigg)^\frac{1}{2}-1\bigg)\!\!\!\left(1+\frac{k_2^2}{k_1^2}\right)^\frac{1}{2}+\left(1+\frac{k_2^2}{k_1^2}\right)^\frac{1}{2}-1
\]
and noting that $(1+\beta |k|^2)|k|^{-1}\tanh|k|-1 \gtrsim |k|$ for $|k| \geq \delta$, one finds that
\[n(k) \gtrsim |k|^\frac{1}{2}\left(1+\frac{k_2^2}{k_1^2}\right)^\frac{1}{2}=\frac{|k|^{\frac{3}{2}}}{|k_1|}\]
and therefore
\begin{eqnarray*}
\left( 1 + \frac{k_2^2}{k_1^2} + \frac{k_2^4}{k_1^2} + |k|^{2s}  \right)n(k)^{-2}
& \lesssim &
\frac{k_1^2}{|k|^3}+\frac{k_2^2}{|k|^3}+\frac{k_2^4}{|k|^3}+k_1^2|k|^{2s-3} \nonumber \\
& \lesssim & |k| + k_1^2|k|^{2s-3}
\end{eqnarray*}
for $|k| \geq \delta$.
On the other hand, in the regime $|k| \leq \delta$, $\left|\frac{k_2}{k_1}\right| \geq \delta$ one has that
\[\afl
\left( 1 + \frac{k_2^2}{k_1^2} + \frac{k_2^4}{k_1^2} + |k|^{2s}  \right)n(k)^{-2}
\lesssim \left(1+\frac{k_2^2}{k_1^2}\right)\!\!\!
\bigg(\bigg(1+\frac{k_2^2}{k_1^2}\bigg)^\frac{1}{2}-1\bigg)^{-2}
\lesssim 1;
\]
altogether we have established that $|n(\Diff)^{-1} (\cdot)|_X^2 \lesssim |\cdot|_Z^2$.
\qed

\begin{remark} \label{rem:m iso}
A straightforward modification of the above proof shows that $m^{-1}(\Diff)$ maps $Z$ continuously into $X$, so that
$m$ is an isomorphism $X \to Z$.
It is however rather the multiplier $n$ that appears in our analysis.
\end{remark}

In view of the KP-scaling $(k_1, k_2) \mapsto (\varepsilon k_1, \varepsilon^2 k_2)$ it is convenient to work with
the scaled norm
\begin{equation}\label{eq:epsilon norm}
|u_1|_\varepsilon^2 = \int_{\R^2} \left( 1 +  \varepsilon^{-2} \frac{k_2^2}{k_1^2} + \varepsilon^{-2} k_1^2\right) |\hat u_1(k)|^2\, \diff k
\end{equation}
for $\tilde{Y}_1$ (or, equivalently, for $\chi(\Diff) L^2(\R^2)$, $X_1$, $Y_1$, $Z_1$).

\begin{lemma}\label{lemma:epsilon}
The estimates
\[|u_1|_{m,\infty} \lesssim \varepsilon |u_1|_\varepsilon, \qquad m \geq 0,\]
and
\[
| u_1 v |_Z \lesssim \varepsilon | u_1 |_\varepsilon |v|_{X}{,}
\] 
hold for all $u_1 \in X_1$ and $v \in X$. 
\end{lemma}
\proof
Note that
\[
| u_1 |_{m,\infty} \lesssim | (1+|k|^{2m}) \hat u_1 |_{L_1(\R^2)} \lesssim  |\hat u_1 |_{L^1}  \leq  | u_1|_{\varepsilon} I^{\frac{1}{2}},
\]
where
\begin{eqnarray*}
I & = & \int_C \frac{1}{1+ \varepsilon^{-2}\tfrac{k_2^2}{k_1^2}+\varepsilon^{-2}k_1^2} \dk \nonumber \\
& = & 4\varepsilon^2 \int_0^{\delta/\varepsilon} \int_0^{1/\varepsilon^2} \frac{k_1}{1+k_3^2+k_1^2} \dk_3\dk_1 \nonumber \\
&   \lesssim & \varepsilon^2.
\end{eqnarray*}
Choosing $m>s$, one therefore finds that
\[
| u_1 v |_Z \lesssim | u_1 v |_H \lesssim | u_1 |_{m,\infty} |v|_H \lesssim \varepsilon | u_1 |_\varepsilon |v|_X.\hspace{1.36in}\Box
\]

Finally, we introduce the space $\tilde{Y}_\varepsilon := \chi_\varepsilon(\Diff)\tilde{Y}$, where
$\chi_\varepsilon(k_1,k_2)=\chi(\varepsilon k_1, \varepsilon^2 k_2)$ (with norm
$|\cdot|_{\tilde Y}$), noting the relationship
\[|u|_\varepsilon^2=\varepsilon|\zeta|_{\tilde{Y}}^2,
\qquad  u(x,y)=\varepsilon^2 \zeta(\varepsilon x, \varepsilon^2 y)
\]
for $\zeta \in \tilde{Y}_\varepsilon$. Observe that $\tilde{Y}_\varepsilon$ coincides with
$\chi_\varepsilon(\Diff)X$, $\chi_\varepsilon(\Diff)Y$, $\chi_\varepsilon(\Diff)Z$ and $\chi_\varepsilon(\Diff)L^2(\R^2)$
for $\varepsilon>0$, while in the limit
$\varepsilon \to 0$ we find that $\tilde{Y}_0 = \tilde{Y}$. We work in particular with the distinguished subsets $\{\zeta: |\zeta|_{\tilde Y} < M\}$
and $\{\zeta: |\zeta|_{\tilde Y} < M-1\}$ of $\tilde{Y}_\varepsilon$, denoting them by respectively $B_M(0)$
and $B_{M-1}(0)$.

We conclude this section with a result which is used in our analysis of the KP-I functional ${\mathcal T}_0$. 

\begin{corollary}
The functional ${\mathcal T}_0$ maps $\tilde{Y}$ smoothly into $\R$ and its critical points are precisely the homoclinic solutions of
equation \eref{eq:normalised steady KP again}.
\end{corollary}

Figure \ref{fig:spaces} summarises the various spaces and their relationships to each other.

\begin{figure}
\centering
\includegraphics{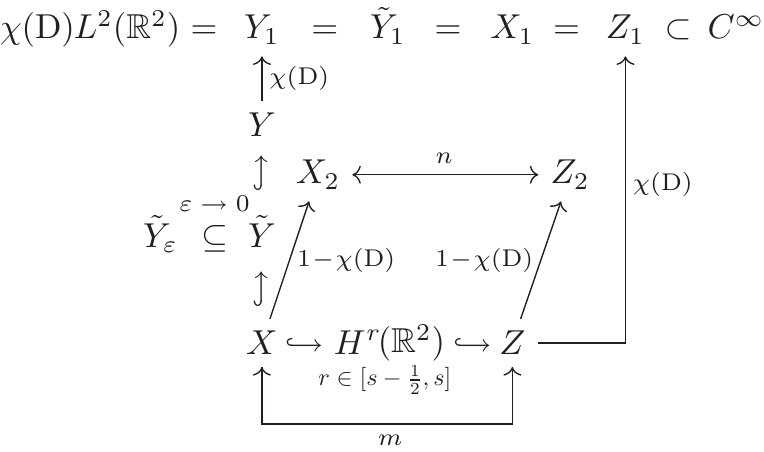}

\caption{\small An overview of the spaces used in this paper. The spaces $Y$ and $\tilde Y$ are the energy spaces for respectively the FDKP-I and KP-I equations.
The operator $\chi(\Diff)$ induces orthogonal decompositions $X=X_1 \oplus X_2$,
$Z=Z_1\oplus Z_2$, while $m(\Diff)$, $n(\Diff)$ define isomorphisms $X \to Z$ and $X_2 \to Z_2$.
Finally, $\tilde{Y}_\varepsilon = \chi(\varepsilon \Diff)\tilde{Y}$.}
\label{fig:spaces}

\end{figure}

\section{Variational reduction} \label{sec:reduction}

We proceed by making the Ansatz $c = 1 - \varepsilon^2$ and seeking critical points of the functional
${\mathcal I}_\varepsilon: X \to \R$ given by
\begin{equation}\label{eq:I_varepsilon}
{ \mathcal{I}_\varepsilon(u) =  \frac{1}{2}\int_{\R^2} \left( \varepsilon^2 u^2 + (n(\Diff)^{\frac{1}{2}} u)^2  \right) \dx \dy + \frac{1}{3} \int_{\R^2} u^3} \dx \dy,
\end{equation}
so that the critical points of ${\mathcal I}_\varepsilon$ are precisely the homoclinic solutions of \eref{eq:steady FDKP} with
$c=1-\varepsilon^2$.

Note that $u = u_1 + u_2 \in X_1 \oplus X_2$ is a critical point of $\mathcal{I}_\varepsilon$ if and only if
\[
\diff {\mathcal I}_\varepsilon[u_1 + u_2](w_1) = 0, \qquad \diff {\mathcal I}_\varepsilon[u_1 + u_2](w_2) = 0
\]
for all $w_1, w_2 \in X$, which equations are equivalent to the system
\begin{eqnarray}
\varepsilon^2 u_1 + n(\Diff) u_1  { +} \chi(\Diff)(u_1 + u_2)^2 =0, & & \qquad \mbox{ in } Z_1,\label{eq:system i}\nonumber \\
\varepsilon^2 u_2 + n(\Diff) u_2  { +} (1-\chi(\Diff))(u_1 + u_2)^2 =0, & & \qquad \mbox{ in } Z_2.\label{eq:system ii}
\end{eqnarray}

The next step is to solve \eref{eq:system ii} for $u_2$ as a function of $u_1$ using the following
result, which is a proved by a straightforward application of the contraction mapping principle. 

\begin{lemma} \label{lemma:fixed-point}
Let $W_1$, $W_2$ be Banach spaces, $\overline{B}_1$ be a closed ball centred on the origin in $W_1$,
$r$ be a continuous function $\overline{B}_1 \to [0,\infty)$
and $F \colon \overline{B}_1 \times W_2 \to W_2$ be a smooth function satisfying
\[
|F(w_1,0)|_{W_2} \leq \tfrac{1}{2}r(w_1), \qquad |\diff_2 F[w_1,w_2]|_{W_2 \to W_2} \leq \tfrac{1}{3}
\]
for all $(w_1,w_2) \in \overline{B}_1 \times \overline{B}_{r(w_1)}(0)$. The fixed-point equation
\[
w_2 = F(w_1, w_2)
\]
has for each $w_1 \in \overline{B}_1$ a unique solution $w_2 = w_2(w_1) \in \overline{B}_{r(w_1)}(0)$. Moreover
$w_2$ is a smooth function of $w_1$ and
satisfies
\[
|\diff w_2[w_1] |_{W_1 \to W_2} \lesssim | \diff_1 F[w_1,w_2] |_{W_1 \to W_2},
\]
and
\begin{eqnarray*}
\afl| \diff^2 w_2[w_1] |_{W_1^2 \to W_2} &\lesssim & |\diff_1^2 F[w_1, w_2] |_{W_1^2 \to W_2}\nonumber \\ 
&&\quad \mbox{}+  | \diff_2 \diff_2 F[w_1,w_2]|_{W_1 \times W_2 \to W_2} \, | \diff_1 F[w_1, w_2] |_{W_1 \to W_2} \nonumber \\
&&\quad \mbox{}+  | \diff_2^2 F[w_1,w_2] |_{W_2^2 \to W_2} \, | \diff_1 F[w_1, w_2]|^2_{W_1 \to W_2}.
\end{eqnarray*}
\end{lemma}

Write \eref{eq:system ii} as 
\begin{equation}\label{eq:eqn with G}
u_2=G(u_1,u_2), 
\end{equation}
where
\begin{equation}
\label{eq:G}
G(u_1,u_2) = -n(\Diff)^{-1} { (1-\chi(\Diff) ) \left( \varepsilon^2 u_2 +  (u_1 + u_2)^2 \right)};
\end{equation}
the following mapping property of $G$ follows from Corollary \ref{cor:spaces corollary} and Proposition \ref{prop:FDKP op weak}.

\begin{proposition} \label{prop:G weak}
Equation \eref{eq:G} defines a smooth and weakly continuous mapping $G:X_1 \times X_2 \to X_2$.
\end{proposition}

\begin{lemma}\label{lemma:u_2}
{  Define $U = \{u_1 \in X_1: |u_1|_\varepsilon \leq 1\}$.}
Equation \eref{eq:eqn with G} defines a map
\[U \ni u_1 \mapsto  u_2(u_1) \in X_2,\]
which satisfies
\[|\diff^k u_2[u_1]|_{X_1^k \to X_2} \lesssim \varepsilon |u_1|_\varepsilon^{2-k}, \qquad k = 0,1,2\]
(where by convention $|\diff^k u_2[u_1]|_{X_1^k \to X_2}$ is interpreted as
$|u_2(u_1)|_\varepsilon$ for $k=0$).
\end{lemma}
\proof
We apply Lemma \ref{lemma:fixed-point} to equation \eref{eq:eqn with G} with $W_1=(X_1,|\cdot|_\varepsilon)$
$W_2=(X_2,|\cdot|_X)$ and $F=G$. Note that
\begin{eqnarray*}
\diff_1 G[u_1,u_2](v_1) &= & -n(D)^{-1}{(1-\chi(\Diff) )(2(u_1+u_2)v_1)}, \nonumber \\
\diff_2 G[u_1,u_2](v_2) &= &- n(D)^{-1}{(1-\chi(\Diff))(\varepsilon^2 v_2 + 2(u_1+u_2)v_2)}
\end{eqnarray*}
and
\[
| (n(\Diff))^{-1} {(1-\chi(\Diff))} z|_X \lesssim |z|_Z
\]
(Lemma \ref{lemma:n}).
Using Lemmata \ref{lemma:global embeddings} and \ref{lemma:epsilon}, we therefore find that
\[
|G(u_1,0)|_X = |u_1^2|_Z \lesssim   \varepsilon  |u_1|_\varepsilon |u_1|_X
\lesssim \varepsilon |u_1|_\varepsilon |u_1|_{L^2} \leq \varepsilon |u_1|_\varepsilon^2
\]
and
\begin{eqnarray*}
|\diff_2  G[u_1,u_2](v_2)|_X
&\lesssim \varepsilon^2 |v_2|_Z + |u_1 v_2|_Z  + |u_2 v_2|_Z\nonumber \\
&\lesssim (\varepsilon^2 + \varepsilon |u_1|_\varepsilon + |u_2|_X) |v_2|_{X}.
\end{eqnarray*}
To satisfy the assumptions of Lemma~\ref{lemma:fixed-point},
we choose $r(u_1)=\sigma \varepsilon |u_1|_\varepsilon^2$ for a sufficiently large value
of $\sigma>0$, so that
\[|u_2|_X \lesssim \tfrac{1}{2}r(u_1), \qquad
 |\diff_2  G[u_1,u_2]|_{X_2 \to X_2}  \lesssim \varepsilon
\]
for $(u_1,u_2)  \in U \times \overline{B}_{r(u_1)}(0)$. The lemma asserts the existence
of a unique solution $u_2(u_1) \in \overline{B}_{r(u_1)}(0)$
of \eref{eq:eqn with G} for each $u_1 \in U$
which satisfies
\[
|u_2(u_1)|_X \lesssim \varepsilon | u_1 |_\varepsilon^2.
\]

Observe that
\begin{eqnarray*}
|\diff_1 G[u_1,u_2] (v_1)|_X &\lesssim & |u_1 v_1|_Z + |u_2 v_1|_Z\nonumber \\  
&\lesssim & \varepsilon  (|u_1|_X + |u_2|_X) |v_1|_\varepsilon\nonumber \\
&\lesssim&  \varepsilon (|u_1|_\varepsilon + \varepsilon |u_1|_\varepsilon^2) |v_1|_\varepsilon,
\end{eqnarray*}
and similarly
\begin{eqnarray*}
|\diff_1^2 G[u_1,u_2] (v_1,w_1)|_X \lesssim |v_1 w_1|_Z \lesssim \varepsilon |v_1|_\varepsilon |w_1|_\varepsilon,\nonumber \\
|\diff_2^2 G[u_1,u_2] (v_2,w_2)|_X \lesssim |v_2 w_2|_Z \lesssim |v_2|_X |w_2|_X,\nonumber \\
|\diff_1 \diff_2 G[u_1,u_2] (v_1,v_2)|_X \lesssim |v_1 v_2|_Z \lesssim \varepsilon |v_1|_\varepsilon |v_2|_X.
\end{eqnarray*}
Combining these estimates in the fashion indicated in Lemma \ref{lemma:fixed-point}, one finds that
\[
|u_1|_\varepsilon^{-2} |u_2(u_1)|_X + |u_1|_\varepsilon^{-1} |\diff u_2[u_1] |_{X_1 \to X_2} + |\diff^2 u_2[u_1] |_{X_1^2 \to X_2} \lesssim \varepsilon.\hspace{0.52cm}\Box
\]

Our next result shows in particular that $u=u_1+u_2(u_1)$ belongs to $H^\infty(\R^2)$
for each $u_1 \in U_1$.

\begin{proposition}
Any function $u=u_1+u_2 \in X_1 \oplus X_2$ which satisfies \eref{eq:eqn with G}
belongs to $H^\infty(\R^2)$.
\end{proposition}
\proof Obviously $u_1 \in H^\infty(\R^2)$, and to show that $u_2$ is also smooth we
indicate the regularity index $s$ in the spaces $X_2$ and $Z_2$ explicitly.
Since $H^{s}(\R^2)$ is an algebra for $s > \frac{3}{2}$ and $X_2^s \hookrightarrow
(1-\chi(D))H^s(\R^2) \hookrightarrow Z_2^{s+\frac{1}{2}}$ (see Lemma
\ref{lemma:global embeddings}(i)), the mapping
\[
X_1 \oplus X_2^{s} \ni (u_1, u_2) \mapsto -(1-\chi(\Diff) ) \left( \varepsilon^2 u_2 +  (u_1 + u_2)^2 \right) \in Z_2^{s+\frac{1}{2}}
\]
is continuous. It follows that $u_2 \in X_2^{s+\frac{1}{2}}$ because $n(D)$ is an isomorphism $X_2^{s+\frac{1}{2}} \to Z_2^{s+\frac{1}{2}}$ (see Lemma \ref{lemma:n}). Bootstrapping this argument yields
$u_2 \in X_2^s \subset H^s(\R^2)$ for any $s \in \R$.\qed

The (smooth)
reduced variational functional ${\mathcal J}_\varepsilon: U \to \R$ is defined by
\begin{eqnarray*}
\afl{\mathcal J}_\varepsilon(u_1) & := & {\mathcal I}_\varepsilon(u_1+u_2(u_1)), \nonumber \\
\afl& = & \frac{1}{2}\int_{\R^2} \left( \varepsilon^2 u_1^2 + \varepsilon^2 u_2(u_1)^2 + (n(\Diff)^{\frac{1}{2}} u_1)^2 +
(n(\Diff)^{\frac{1}{2}} u_2(u_1))^2  \right) \dx \dy \nonumber \\
\afl && \qquad\mbox{} + \frac{1}{3} \int_{\R^2} (u_1+u_2(u_1))^3 \dx \dy
\end{eqnarray*}
(recall that $\langle u_1, u_2(u_1) \rangle_{L^2}=0)$),
where $\diff {\mathcal I}_\varepsilon[u_1+u_2(u_1)](v_2)=0$ for all $v_2 \in X_2$ by construction.
It follows that
\begin{eqnarray*}
\diff {\mathcal J}_\varepsilon[u_1](v_1) & = &\diff {\mathcal I}_\varepsilon[u_1+u_2(u_1)](v_1)+\diff {\mathcal I}_\varepsilon[u_1+u_2(u_1)](\diff u_2[u_1](v_1))\nonumber \\
& = & \diff {\mathcal I}_\varepsilon[u_1+ u_2(u_1)](v_1)
\end{eqnarray*}
for all $v_1 \in X_1$, so that each critical point $u_1$ of ${\mathcal J}_\varepsilon$ defines a critical point $u_1+u_2(u_1)$ of ${\mathcal I}_\varepsilon$.
Conversely, each critical point $u=u_1+u_2$ of ${\mathcal I}_\varepsilon$ with $u_1 \in U$
has the properties
that $u_2=u_2(u_1)$ and $u_1$ is a critical point of ${\mathcal J}_\varepsilon$.

\begin{lemma}\label{lemma:tilde J}
The reduced functional ${\mathcal J}_\varepsilon: U \to \R$ satisfies
\[
\afl
{\mathcal J}_\varepsilon(u_1) = \frac{1}{2}\int_{\R^2} \left( \varepsilon^2 u_1^2+ (n(\Diff)^{\frac{1}{2}} u_1)^2  \right) \dx \dy + \frac{1}{3} \int_{\R^2} u_1^3 \dx \dy + \mathcal{R}_\varepsilon(u_1),
\]
where
\[|\diff^k \mathcal R_\varepsilon(u_1)|_{X_1^k \to \R} \lesssim \varepsilon^2 |u_1|_\varepsilon^{4-k}, \qquad
k=0,1,2.\]
\end{lemma}
\proof
Observe that
\[\mathcal{R}_\varepsilon(u_1)=\tfrac{1}{2}\varepsilon^2 K_1(u_1)+K_2(u_1)+K_3(u_1)+\tfrac{1}{3}K_4(u_1)+\tfrac{1}{2}K_5(u_1),\]
where
\[
\afl K_1(u_1) = |u_2(u_1)|_{L^2}^2, \quad
K_2(u_1) = \langle u_1^2, u_2(u_1) \rangle_{L^2}, \quad K_3(u_1) = \langle u_1 u_2(u_1), u_2(u_1)\rangle_{L^2},
\]
\[K_4(u_1) = \langle u_2(u_1)^2, u_2(u_1) \rangle_{L^2}, \quad
K_5(u_1) = | (n(\Diff))^{\frac{1}{2}} u_2(u_1)|_{L^2}^2.
\]
We investigate each of these quantities using the estimates
\begin{eqnarray*}
{ 
\afl |\diff^j u_2 [u_1]({\bf v})|_{\mathrm{BC} \cap L^2}
}
& \lesssim & |\diff^j u_2 [u_1]({\bf v})|_{X} \lesssim \varepsilon |u_1|_\varepsilon^{2-j} |v_1|_\varepsilon \cdots |v_j|_{\varepsilon}, \qquad j=0,1,2,\nonumber \\
\afl \hspace{8mm}|\diff^j u_1[u_1]({\bf v})|_{\mathrm{BC}} &\lesssim & \varepsilon |u_1|_\varepsilon^{1-j} |v_1|_\varepsilon \cdots |v_j|_\varepsilon, \qquad j=0,1,\nonumber \\
\afl \hspace{8mm}|\diff^j u_1[u_1]({\bf v})|_{L^2} &\lesssim & |u_1|_\varepsilon^{1-j} |v_1|_\varepsilon \cdots |v_j|_\varepsilon,
\qquad j=0,1,
\end{eqnarray*}
where ${\bf v} = (v_1, \ldots, v_j)$ denotes a general element in $X_1^j$, and of course
\[\diff^2 u_1 [u_1] = 0.\]

Using Leibniz's rule, H\"{o}lder's inequality and the basic estimate
$|\langle w \cdot,\cdot \rangle_{L^2}| \leq |w|_{\mathrm{BC}} |\langle \cdot,\cdot \rangle_{L^2}|$,
one finds that
\begin{eqnarray*}
\afl\left| \diff^k K_1[u_1]({\bf v}) \right| &\lesssim \sum_{j=0}^k \left|  \langle \diff^j u_2 [u_1]({\bf v}), \diff^{k-j} u_2 [u_1]({\bf v}) \rangle_{L^2} \right|  \nonumber \\
&\lesssim \sum_{j=0}^k   |\diff^j u_2 [u_1]({\bf v})|_{L^2}  |\diff^{k-j} u_2 [u_1]({\bf v})|_{L^2}\nonumber \\
&\lesssim  \varepsilon^2 \sum_{j=0}^k  |u_1|_{\varepsilon}^{2-j} |u_1|_{\varepsilon}^{2-(k-j)} |v_1|_{\varepsilon} \cdots  |v_k|_{\varepsilon}, 
\end{eqnarray*}
\begin{eqnarray*}
\afl\left| \diff^k K_2[u_1]({\bf v}) \right| &\lesssim  \sum_{0 \leq j + l \leq k} \left|  \langle \diff^j u_1 [u_1]({\bf v}) \diff^l u_1 [u_1]({\bf v}), \diff^{k-j -l} u_2 [u_1]({\bf v}) \rangle_{L^2} \right|  \nonumber \\
&\lesssim   \varepsilon\sum_{0 \leq j + l \leq k} |u_1|_\varepsilon^{1-j}   | \diff^l u_1 [u_1]({\bf v})|_{L^2} |\diff^{k-j -l} u_2 [u_1]({\bf v})|_{L^2}  |v_1|_\varepsilon \cdots |v_j|_{\varepsilon} \nonumber \\
&\lesssim \varepsilon^3 |u_1|_\varepsilon^{4-k}  |v_1|_\varepsilon \cdots |v_k|_{\varepsilon},
\end{eqnarray*}
\begin{eqnarray*}
\afl\left| \diff^k K_3[u_1]({\bf v}) \right| &\lesssim  \sum_{0 \leq j + l \leq k} \left|  \langle \diff^j u_1 [u_1]({\bf v}) \diff^l u_2 [u_1]({\bf v}), \diff^{k-j -l} u_2 [u_1]({\bf v}) \rangle_{L^2} \right|  \nonumber \\
&\lesssim   \varepsilon\sum_{0 \leq j + l \leq k} |u_1|_\varepsilon^{1-j}   | \diff^l u_1 [u_1]({\bf v})|_{L^2} |\diff^{k-j -l} u_2 [u_1]({\bf v})|_{L^2}  |v_1|_\varepsilon \cdots |v_j|_{\varepsilon} \nonumber \\
&\lesssim \varepsilon^3 |u_1|_\varepsilon^{5-k}  |v_1|_\varepsilon \cdots |v_k|_{\varepsilon}
\end{eqnarray*}
and
\begin{eqnarray*}
\afl\left| \diff^k K_4[u_1]({\bf v}) \right| &\lesssim  \sum_{0 \leq j + l \leq k} \left|  \langle \diff^j u_2 [u_1]({\bf v}) \diff^l u_2 [u_1]({\bf v}), \diff^{k-j -l} u_2 [u_1]({\bf v}) \rangle_{L^2} \right|  \nonumber \\
&\lesssim   \varepsilon\sum_{0 \leq j + l \leq k} |u_1|_\varepsilon^{2-j}   | \diff^l u_2 [u_1]({\bf v})|_{X} |\diff^{k-j -l} u_2 [u_1]({\bf v})|_X  |v_1|_\varepsilon \cdots |v_j|_{\varepsilon} \nonumber \\
&\lesssim \varepsilon^3 |u_1|_\varepsilon^{6-k} |v_1|_\varepsilon \cdots |v_k|_{\varepsilon}
\end{eqnarray*}
for $k=0,1,2$.

Finally, since $u_2$ solves \eref{eq:system ii}, one obtains
\begin{eqnarray*}
K_5(u_1) &= | (n(\Diff))^{\frac{1}{2}} u_2(u_1)|_{L^2}^2\nonumber \\ 
&= \langle n(\Diff) u_2(u_1), u_2(u_1) \rangle_{L^2}\nonumber \\ 
&= - \varepsilon^2 |u_2(u_1)|_{L^2}^2 - \langle (u_1 + u_2(u_1))^2, u_2(u_1) \rangle_{L^2}\nonumber \\
&= - \varepsilon^2 K_1(u_1)  - K_2(u_1) -  2K_3(u_1) -  K_4(u_1),
\end{eqnarray*}
all of which terms have been estimated.
\qed

The next step is to convert ${\mathcal J}_\varepsilon$ into a perturbation of the KP-I
functional, the main issue being the replacement of $n(k)$ by $\tilde{n}(k)$.

\begin{proposition}\label{prop:I1}
The Fourier multiplier $(n/\tilde{n})^\frac{1}{2}$ defines an isomorphism\linebreak
$I_1:\chi(\Diff)L^2(\R^2) \to \chi(\Diff)L^2(\R^2)$ for sufficiently small values of $\delta$.
\end{proposition}
\proof
Using the elementary estimates
\[n(k)=\tilde{n}(k) + \bigO(|(k_1,\tfrac{k_2}{k_1})|^4), \qquad
\tilde{n}(k) \eqsim |(k_1,\tfrac{k_2}{k_1})|^2\]
as $(k_1,\frac{k_2}{k_1}) \to 0$, we find that
\[\left|\frac{n(k)}{\tilde n(k)}-1\right| \lesssim \tilde n(k) \lesssim \delta^2,\]
and hence
\[\left(\frac{n(k)}{\tilde{n}(k)}\right)^{\frac{1}{2}} \eqsim 1\]
for $k \in C$, for sufficiently small values of $\delta$.
\qed

We now express the reduced functional in terms of $\tilde u_1=(\frac{n}{\tilde{n}})^\frac{1}{2}u_1$;
to this end define
$\widetilde{\mathcal J}_\varepsilon(\tilde{u}_1)={\mathcal J}_\varepsilon(u_1(\tilde{u}_1))$
and note that $\widetilde{\mathcal J}_\varepsilon$ is a smooth functional $\tilde{U} \to \R$,
where $\tilde{U}=\{u_1 \in X_1: |u_1|_{\varepsilon} \leq \tilde{\tau}\}$ and $\tilde{\tau} \in (0,1)$ is chosen so that 
$\tilde{U} \subseteq I_1[U]$.

 \begin{lemma}\label{lemma:cov to KP}
The reduced functional $\widetilde{\mathcal J}_\varepsilon: \tilde{U} \to \R$ satisfies
\[
\afl\tilde{\mathcal{J}}_\varepsilon(\tilde u_1) = \frac{1}{2}\int_{\R^2} \left( \varepsilon^2 \tilde u_1^2+ (\tilde n(\Diff)^{\frac{1}{2}} \tilde u_1)^2  \right) \dx \dy { +} \frac{1}{3} \int_{\R^2} \tilde u_1^3 \dx \dy  + \tilde{\mathcal{R}}_\varepsilon(\tilde u_1),
\]
where
\[
|\diff^k \tilde{\mathcal R}_\varepsilon[\tilde u_1]|_{X_1^k \to \R} \lesssim \varepsilon^2 |u_1|_\varepsilon^{3-k} + \varepsilon^4 |u_1|_\varepsilon^{2-k}, \qquad k = 0,1,2.\]
 \end{lemma}
\proof 
By construction
\[\int_{\R^2} (n(D)^\frac{1}{2}u_1)^2 \dx\dy = \int_{\R^2} (\tilde{n}(D)^\frac{1}{2}\tilde{u}_1)^2 \dx\dy,\]
and furthermore
\[
|\diff^k {\mathcal R}_\varepsilon[\tilde u_1]|_{X_1^k \to \R} \lesssim \varepsilon^2 |\tilde u_1|_\varepsilon^{4-k},
\qquad k=0,1,2,
\]
because $u_1 \mapsto \tilde{u}_1$ is an isomorphism $X_1 \to X_1$ (here we have abbreviated
${\mathcal R}_\varepsilon(u_1(\tilde{u}_1))$ to
${\mathcal R}_\varepsilon(\tilde{u}_1)$). It remains to estimate the differences
\[\int_{\R^2} u_1^2\dx\dy  - \int_{\R^2} \tilde u_1^2\dx\dy, \qquad
\int_{\R^2} u_1^3\dx\dy -  \int_{\R^2} \tilde u_1^3\dx\dy\]
using the formulae
\[\afl u_1(\tilde{u}_1) = \left( \frac{\tilde n(\Diff)}{n(\Diff)}\right)^{\frac{1}{2}} \tilde{u}_1, \qquad
\diff u_1[\tilde{u}_1](v_1) = \left( \frac{\tilde n(\Diff)}{n(\Diff)}\right)^{\frac{1}{2}} v_1, \qquad
\diff^2 u_1[\tilde{u}_1]=0.\]

Observe that
\begin{eqnarray*}
\left| \left|1- \frac{n(\Diff)}{\tilde n(\Diff)}\right|^{\frac{1}{2}} w_1 \right|_{L^2}^2
&\lesssim | \tilde n(D)^\frac{1}{2} w_1|_{L^2}^2 \lesssim \varepsilon^2 |w_1|_{\varepsilon}^2,
\end{eqnarray*}
for $w_1 \in \chi(\Diff) L^2(\R^2)$, so that
\[
\left| \left|1- \frac{n(\Diff)}{\tilde n(\Diff)}\right|^{\frac{1}{2}} \diff^j u_1[\tilde{u}_1]({\bf v}) \right|_{L^2} 
\lesssim
\varepsilon |\tilde{u}_1|_\varepsilon^{1-j} |v_1|_\varepsilon \cdots |v_j|_\varepsilon, \qquad j=0,1.
\]
It follows that
 \[
 K_6(\tilde u_1) := \int_{\R^2} \left( u_1^2 - \tilde u_1^2 \right) \dx \dy = \int_{\R^2} 
 \left(\left|1- \frac{n(\Diff)}{\tilde n(\Diff)}\right|^{\frac{1}{2}}u_1\right)^2 \dx\dy
 \] 
satisfies
 \begin{eqnarray*}
\afl\left| \diff^k K_6[\tilde u_1]({\bf v}) \right| &\leq &  \sum_{j=0}^k \left| \left\langle \left |1- \frac{n(D)}{\tilde n(D)}\right|^{\frac{1}{2}} \, \diff^j u_1[\tilde{u}_1]({\bf v}), \left|1-  \frac{n(\Diff)}{\tilde n(\Diff)}\right|^{\frac{1}{2}} \, \diff^{k-j} u_1[\tilde{u}_1]({\bf v}) \right\rangle_{L^2} \right|\nonumber \\ 
\afl &\lesssim &\varepsilon^2  |\tilde u_1|_\varepsilon^{2-k} |v_1|_\varepsilon \cdots |v_k|_\varepsilon, \qquad k=0,1,2.
\end{eqnarray*}
The term
\[
K_7(\tilde u_1) := \int_{\R^2} \left( u_1^3 - \tilde u_1^3 \right) \dx \dy =  \sum_{m=0}^2 \int_{\R^2} ( u_1 - \tilde u_1)  u_1^m \tilde u_1^{2-m} \dx \dy
\]
is treated in a similar fashion. Using the estimate $|v_1|_{\mathrm{BC}} \lesssim \varepsilon |v_1|_\varepsilon$
(see Lemma~\ref{lemma:epsilon}), we find that
\begin{eqnarray*}
\afl\lefteqn{\sum_{j + l = k}\left|  \left\langle  \left( 1 - \left( \frac{n(\Diff)}{\tilde n(\Diff)} \right)^{\frac{1}{2}} \right) \diff^j u_1[\tilde{u}_1]({\bf v}) , \diff^l \tilde u_1 [\tilde{u}_1]({\bf v}) \,  \diff^{k-l-j} \tilde u_1 [\tilde{u}_1]({\bf v}) \right\rangle  \right|_{L^2}}\nonumber \\
\afl \quad &\lesssim & \sum_{j + l = k} |\diff^l \tilde u_1 [\tilde{u}_1]({\bf v})|_{\mathrm{BC}}  \left|  \left( 1 - \left( \frac{n(\Diff)}{\tilde n(\Diff)} \right)^{\frac{1}{2}} \right) \diff^j u_1[\tilde{u}_1]({\bf v}) \right|_{L^2} \left| \diff^{k-j-l} \tilde u_1 [\tilde{u}_1]({\bf v})   \right|_{L^2}\nonumber \\
 \afl&\lesssim &  \sum_{j + l = k} \varepsilon |\tilde u_1|_\varepsilon^{1-l}  \varepsilon |\tilde u_1|_\varepsilon^{1-j} |\tilde u_1|_\varepsilon^{1-(k-j-l)} |v_1|_{\varepsilon} \cdots |v_k|_\varepsilon, \qquad k=0,1,2.
\end{eqnarray*}
with similar estimates for the summands with $m=0$ and $m=2$ in the formula for $K_7(\tilde{u})$.
Altogether we find that
\[|\diff^k K_7[\tilde u_1]({\bf v})| \lesssim \varepsilon^2 | \tilde u_1|_\varepsilon^{3-k} |v_1|_\varepsilon \cdots |v_k|_\varepsilon, \qquad k=0,1,2.\hspace{1.8cm}\Box\]

\begin{remark}
Using the simple expansion $n(k)=\tilde{n}(k) +\bigO(|(k_1,\tfrac{k_2}{k_1})|^4)$ for $k \in C$ 
leads to the insufficient estimate
\[\afl\int_{\R^2} \left((n(D)-\tilde{n}(D))^\frac{1}{2}u_1\right)^2 \dx\dy=
\int_{\R^2} |n(k)-\tilde{n}(k)||\hat{u}_1|^2 \dk = \bigO(\varepsilon^2 |u_1|_\varepsilon^2)\]
(at the next step we use the KP scaling for $u$ and scale the functional by
$\varepsilon^{-3}$).
\end{remark}

Finally, we use the KP-scaling
\[
 \tilde{u}_1(x,y) = \varepsilon^2 \zeta(\varepsilon x,\varepsilon^2 y).
\]
{ The following proposition is immediate.}

\begin{proposition} \label{prop:I2}
The mapping $\tilde{u}_1 \mapsto \zeta$ defines an isomorphism $I_2: \chi(D)L^2(\R) \to \chi_\varepsilon(D)L^2(\R)$.
\end{proposition}

Note that $\tilde{u}_1 \in \chi(D)L^2(\R^2)$ has $\supp \hat{\tilde{u}}_1 \in C$, while
$\chi_\varepsilon(D) L^2(\R^2)$ has $\supp \hat{\zeta} \in C_\varepsilon$, where
\[C_\varepsilon=\left\{(k_1,k_2): |k| \leq \tfrac{\delta}{\varepsilon}, \left|\tfrac{k_2}{k_1}\right|\leq \tfrac{\delta}{\varepsilon}\right\}.\]
The formula
${\mathcal T}_\varepsilon(\zeta):=\varepsilon^{-3}\widetilde{\mathcal J}_\varepsilon(\tilde{u}_1(\zeta))$
therefore defines a smooth functional $B_M(0) \to \R$, where $B_M(0)=\{\zeta \in \tilde{Y}_\varepsilon: |\zeta|<M\}$
and $M>1$ is chosen so that $B_M(0) \subseteq I_2[\tilde{U}]$, that is
$M \lesssim \varepsilon^{-\frac{1}{2}}\tilde{\tau}$. (Recall that 
$\tilde Y_\varepsilon = \chi_\varepsilon(D)\tilde{Y}$ consists of those functions in $\tilde{Y}$
whose Fourier transforms are supported in $C_\varepsilon$; for $\varepsilon = 1$ it coincides with $\tilde Y_1$, and in the limit $\varepsilon \to 0$ it `fills out' all of $\tilde Y$.)

Using Lemma \ref{lemma:cov to KP} and the calculations
\[
|\varepsilon \tilde u_1|_{L^2}^2 + | \tilde n(\Diff)^{\frac{1}{2}} \tilde u_1|_{L^2}^2 = \varepsilon^{3} |\zeta|_{\tilde Y}^2,
\qquad
| \tilde u_1 |_\varepsilon = \varepsilon^{\frac{1}{2}} |\zeta|_{\tilde Y},
\]
one finds that
\begin{equation}
{\mathcal T}_\varepsilon(\zeta)={\mathcal Q}(\zeta) { +} {\mathcal S}(\zeta)+\varepsilon^{\frac{1}{2}}{\mathcal R}_\varepsilon(\zeta),
\label{eq:final red func}
\end{equation}
where
\[
{\mathcal Q}(\zeta)=\frac{1}{2}|\zeta|_{\tilde Y}^2, \qquad 
{\mathcal S}(\zeta)=\frac{1}{3}\int_{\R^2} \zeta^3 \dx \dy,\]
and
\[
|\diff^k {\mathcal R}_\varepsilon[\zeta]|_{\tilde{Y}^k \to \R} \lesssim |\zeta|_{\tilde Y}^{2-k}, \qquad k = 0,1,2;\]
in particular, we note that
\[
|{\mathcal R}_\varepsilon(\zeta)| + |\diff {\mathcal R}_\varepsilon[\zeta](\zeta)|
+ |\diff^2 {\mathcal R}_\varepsilon[\zeta](\zeta,\zeta)| \lesssim | \zeta|_{\tilde Y}^2.
\]

Let $\zeta \in B_M(0)$ and define $u=u_1(\tilde{u}_1(\zeta)) + u_2(u_1(\tilde{u}_1(\zeta)))$.
By construction one has that
\begin{equation}
\diff {\mathcal I}_\varepsilon[u](w_1)= \varepsilon^3 \diff  {\mathcal T}_\varepsilon[\zeta] (\varrho), \qquad
\diff {\mathcal I}_\varepsilon[u](w_2)=0 \label{eq:I vs T}
\end{equation}
for each $w=w_1+w_2 = X_1 \oplus X_2$, where $\rho=I_2(I_1(w_1))$ (see Propositions \ref{prop:I1} and
\ref{prop:I2}),
so that in particular each critical point $\zeta_\infty$ of ${\mathcal T}_\varepsilon$ defines a critical point
$u_\infty=u_1(\tilde{u}_1(\zeta_\infty)) + u_2(u_1(\tilde{u}_1(\zeta_\infty)))$
of ${\mathcal I}_\varepsilon$. Equation \eref{eq:I vs T} also shows that each
{ Palais--Smale} sequence $\{\zeta_n\}$ for ${\mathcal T}_\varepsilon$ generates a { Palais--Smale}
sequence $\{u_n\}$ with $u_n=u_1(\tilde{u}_1(\zeta_n)) + u_2(u_1(\tilde{u}_1(\zeta_n)))$ for ${\mathcal I}_\varepsilon$,
and our next result confirms that weakly convergent sequences in $B_M(0) \subseteq \tilde{Y}_\varepsilon$ generate
sequences which are weakly convergent in $X$.

\begin{proposition} \label{prop:weak trace back}
Suppose that $\{\zeta_n\} \subset B_M(0)$ converges weakly in $\tilde{Y}_\varepsilon$ to
$\zeta_\infty \in B_M(0)$. The corresponding sequence $\{u_n\}$, where
\[
u_n = u_1(\tilde{u}_1(\zeta_n)) + u_2(u_1(\tilde{u}_1(\zeta_n))),
\]
converges weakly in
$X$  to $u_\infty = u_1(\tilde{u}_1(\zeta_\infty)) + u_2(u_1(\tilde{u}_1(\zeta_\infty)))$.
\end{proposition}
\proof
Abbreviating $u_1(\tilde{u}_1(\zeta_n))$ to $u_{1,n}$, note that
$\{u_{1,n}\} \subset U$ converges weakly in $X_1$ to
$u_{1,\infty}=u_1(\tilde{u}_1(\zeta_\infty)) \in U$.
Furthermore, $u_{2,n}=u_2(u_{1,n})$ is the
unique solution in $X_2$ of equation \eref{eq:eqn with G} with $u_1=u_{1,n}$, so that
\[
u_{2,n}=G(u_{1,n},u_{2,n}).\]

Observe that $\{u_{2,n}\}$ is bounded in $X_2$; the following argument shows that any 
weakly convergent subsequence of $\{u_{2,n}\}$ has weak limit $u_2(u_{1,\infty})$,
so that $\{u_{2,n}\}$ itself converges weakly to $u_2(u_{1,\infty})$ in $X_2$.
Suppose that (a subsequence of) $\{u_{2,n}\}$ converges weakly in $X_2$ to $u_{2,\infty}$.
Because
$G: X_1 \times X_2 \to X_2$ is weakly continuous (Proposition \ref{prop:G weak}), we find that
\[
u_{2,\infty}=G(u_{1,\infty},u_{2,\infty}),
\]
so that $u_{2,\infty} = u_2(u_{1,\infty})$ (the fixed-point equation
$u_2 = G(u_{1,\infty}, u_2)$ has a unique solution in $X_2$).

Altogether we conclude that $\{u_{1,n}+u_{2,n}\}$ converges weakly in $X$ to
$u_\infty=u_{1,\infty} + u_{2,\infty}$.\qed

\section{Existence theory} \label{sec:existence}

The functional ${\mathcal T}_\varepsilon: B_M(0) \to \R$ may be considered as a perturbation
of the `limiting' functional ${\mathcal T}_0: \tilde{Y} \to \R$ with
\[{\mathcal T}_0(\zeta)={\mathcal Q}(\zeta) { +} {\mathcal S}(\zeta).\]
More precisely
$\varepsilon^\frac{1}{2}{\mathcal R}_\varepsilon \circ\chi_\varepsilon(D)$ (which coincides with
${\varepsilon^\frac{1}{2}\mathcal R}_\varepsilon$ on $B_M(0) \subset \tilde{Y}_\varepsilon$) converges uniformly
to zero over $B_M(0) \subset \tilde{Y}$, and corresponding statements for its derivatives also hold.
In this section we study ${\mathcal T}_\varepsilon$ by perturbative arguments in this spirit,
choosing $M>1$ sufficiently large that inequality \eref{eq:choose M} below holds for
some $\zeta_0\in \tilde{Y}\setminus\{0\}$.

We seek critical points of ${\mathcal T}_\varepsilon$ by considering its \emph{natural constraint set}
\[
N_\varepsilon = \left\{ \zeta \in B_M(0) \colon \zeta \neq 0, \diff {\mathcal T}_\varepsilon[\zeta](\zeta) = 0 \right\},
\]
noting the calculation
\begin{equation}\label{eq:dT}
\diff {\mathcal T}_\varepsilon[\zeta](\zeta)  = 2 {\mathcal Q}(\zeta) { +} 3 {\mathcal S}(\zeta) + \varepsilon^{\frac{1}{2}} \diff {\mathcal R}_\varepsilon[\zeta](\zeta),
\end{equation}
which shows that
\begin{eqnarray*}
{ -}{\mathcal S}(\zeta) & = & \tfrac{2}{3}{\mathcal Q}(\zeta)
+ \tfrac{1}{3}\varepsilon^{\frac{1}{2}} \diff {\mathcal R}_\varepsilon[\zeta](\zeta) \nonumber \\
& = & \tfrac{1}{3}|\zeta|_{\tilde Y}^2 + \bigO(\varepsilon^\frac{1}{2}|\zeta|_{\tilde Y}^2) \nonumber \\
& \geq & \tfrac{1}{6}|\zeta|_{\tilde Y}^2
\end{eqnarray*}
and
\begin{eqnarray*}
\diff^2 {\mathcal T}_\varepsilon[\zeta](\zeta, \zeta) &=& 2 {\mathcal Q}(\zeta) { +} 6 {\mathcal S}(\zeta) + \varepsilon^{\frac{1}{2}} \diff^2 {\mathcal R}_\varepsilon[\zeta](\zeta, \zeta) ^{\frac{1}{2}}\nonumber \\
&=& - 2 {\mathcal Q}(\zeta) - 2\varepsilon^{\frac{1}{2}} \diff {\mathcal R}_\varepsilon[\zeta](\zeta) + \varepsilon^{\frac{1}{2}} \diff^2 {\mathcal R}_\varepsilon[\zeta](\zeta, \zeta) \nonumber \\
&= &- |\zeta|_{\tilde Y}^2 +  \bigO (\varepsilon^{\frac{1}{2}}|\zeta|_{\tilde Y}^2) \nonumber \\
&\leq& -\tfrac{1}{2} |\zeta|_{\tilde Y}^2
\end{eqnarray*}
(and in particular ${\mathcal S}(\zeta) { <} 0$, $\diff^2 {\mathcal T}_\varepsilon[\zeta](\zeta, \zeta)<0$)
for points $\zeta \in N_\varepsilon$. 
Any nontrivial critical point of ${\mathcal T}_\varepsilon$ clearly lies
on $N_\varepsilon$, and the following proposition shows that the converse is also true.

\begin{proposition}
Any critical point of ${\mathcal T}_\varepsilon|_{N_\varepsilon}$
is a (necessarily nontrivial) critical point of ${\mathcal T}_\varepsilon$.
\end{proposition}
\proof
Define ${\mathcal G}_\varepsilon: U_\varepsilon\setminus \{0\} \to \R$ by ${\mathcal G}_\varepsilon(\zeta)=\diff {\mathcal T}_\varepsilon[\zeta](\zeta)$, so that $N_\varepsilon = {\mathcal G}_\varepsilon^{-1}(0)$ and $\diff {\mathcal G}_\varepsilon[\zeta]$
does not vanish on $N_\varepsilon$ (since $\diff {\mathcal G}_\varepsilon[\zeta](\zeta)=\diff ^2{\mathcal T}_\varepsilon[\zeta](\zeta,\zeta)<0$ for $\zeta \in N_\varepsilon$). There exists a Lagrange multiplier $\mu$ such that
\[\diff {\mathcal T}_\varepsilon[\zeta^\star]-\mu\diff {\mathcal G}_\varepsilon[\zeta^\star]=0,\]
and applying this operator to $\zeta^\star$ we find that $\mu = 0$, whence $\diff {\mathcal T}_\varepsilon[\zeta^\star]  = 0$.\qed

There is a convenient geometrical interpretation of $N_\varepsilon$ (see Figure \ref{fig:ncs geometry}).

\begin{proposition} \label{prop:nc interpretation}
Any ray in $(B _M(0)\setminus \{0\}) \cap {\mathcal S}^{-1}{(-\infty,0)} \subset \tilde{Y}_\varepsilon$ intersects
$N_\varepsilon$ in at most one point and the value of ${\mathcal T}_\varepsilon$ along such a ray attains a strict maximum at this point.
\end{proposition}
\proof
Let $\zeta\in (B _M(0)\setminus \{0\}) \cap {\mathcal S}^{-1}{(-\infty,0)} \subset \tilde{Y}_\varepsilon$ and consider the value of ${\mathcal T}_\varepsilon$
along the ray in $B _M(0)\setminus \{0\}$ through $\zeta$, that is, the set
$\{\lambda\zeta:0<\lambda<M/|\zeta|_1\} \subset \tilde{Y}_\varepsilon$.
The calculation
\[
\frac{\diff }{\diff \lambda} {\mathcal T}_\varepsilon(\lambda \zeta)=\diff {\mathcal T}_\varepsilon[\lambda\zeta](\zeta)=\lambda^{-1}\diff {\mathcal T}_\varepsilon[\lambda\zeta](\lambda\zeta)
\]
shows that $\frac{\diff }{\diff \lambda} {\mathcal T}_\varepsilon(\lambda \zeta)=0$ if and only if
$\lambda \zeta \in N_\varepsilon$; furthermore
\begin{eqnarray*}
\frac{\diff ^2}{\diff \lambda^2}{\mathcal T}_\varepsilon(\lambda\zeta)
&=&
2{\mathcal Q}(\zeta)+6\lambda^{-2} {\mathcal S}(\lambda \zeta) +
\varepsilon^{\frac{1}{2}}\diff ^2{\mathcal R}_\varepsilon[\lambda\zeta](\zeta,\zeta)\nonumber \\
&=&
-2{\mathcal Q}(\zeta)
-2\lambda^{-2}\varepsilon^{\frac{1}{2}}\diff {\mathcal R}_\varepsilon[\lambda\zeta](\lambda\zeta)
+\varepsilon^{\frac{1}{2}}\diff ^2{\mathcal R}_\varepsilon[\lambda\zeta](\zeta,\zeta)
\nonumber \\
&=&
-{ 2} {\mathcal Q}(\zeta)+\bigO(\varepsilon^{\frac{1}{2}} |\zeta|_{\tilde Y}^2)\nonumber \\
&<&0
\end{eqnarray*}
for each $\zeta$ with $\lambda \zeta \in N_\varepsilon$.
\qed

 \begin{remark}
If $\varepsilon=0$ we may take $M=\infty$, and in this case 
every ray in ${\mathcal S}^{-1}{(-\infty,0)}$ intersects $N_0$ in precisely one point.
\end{remark}

In view of the above characterisation of nontrivial critical points of ${\mathcal T}_\varepsilon$ we
proceed by seeking a `ground state', that is, a minimiser $\zeta^\star$
of ${\mathcal T}_\varepsilon$ over $N_\varepsilon$. We make frequent use of the identities
\begin{eqnarray}
{\mathcal T}_\varepsilon(\zeta)
&=&\tfrac{1}{3} {\mathcal Q}(\zeta)+\tfrac{1}{3}\diff {\mathcal T}_\varepsilon[\zeta](\zeta)
+\varepsilon^{\frac{1}{2}}\left({\mathcal R}_\varepsilon(\zeta)-\tfrac{1}{3}\diff {\mathcal R}_\varepsilon[\zeta](\zeta)\right), \label{eq:only Q} \\
{\mathcal T}_\varepsilon(\zeta) &=&
{ -} \tfrac{1}{2}{\mathcal S}(\zeta)+\tfrac{1}{2}\diff {\mathcal T}_\varepsilon[\zeta](\zeta)
+\varepsilon^{\frac{1}{2}}\left({\mathcal R}_\varepsilon(\zeta)
-\tfrac{1}{2}\diff {\mathcal R}_\varepsilon[\zeta](\zeta)\right), \label{eq:only S}
\end{eqnarray}
which are obtained using \eref{eq:dT} to eliminate respectively ${\mathcal S}(\zeta)$ and ${\mathcal Q}(\zeta)$
from \eref{eq:final red func}, beginning with some
\emph{a priori} bounds for ${\mathcal T}_\varepsilon|_{N_\varepsilon}$.

\begin{proposition}\label{prop:lower bounds}
Each $\zeta \in N_\varepsilon$ satisfies
${\mathcal T}_\varepsilon(\zeta) \geq \tfrac{1}{12} |\zeta|_{\tilde Y}^2$ and $|\zeta|_{\tilde Y} \gtrsim 1$.
{ In particular}, each
$\zeta\in N_\varepsilon$ with
${\mathcal T}_\varepsilon(\zeta)<\frac{1}{12}(M-1)^2$ satisfies
$|\zeta|_{\tilde Y}<M-1$.
\end{proposition}
\proof
Let $\zeta \in N_\varepsilon$. Using \eref{eq:only Q}, one finds that
\[
{\mathcal T}_\varepsilon(\zeta) = \tfrac{1}{3} {\mathcal Q}(\zeta) +  \bigO(\varepsilon^{\frac{1}{2}}|\zeta|_{\tilde Y}^2)
 = \tfrac{1}{6}|\zeta|_{\tilde Y}^2 +  \bigO(\varepsilon^{\frac{1}{2}}|\zeta|_{\tilde Y}^2)
 \geq \tfrac{1}{12}|\zeta|_{\tilde Y}^2,
\]
so that in particular ${\mathcal T}_\varepsilon(\zeta)<\frac{1}{12}(M-1)^2$ implies that
$|\zeta|_{\tilde Y} < M-1$. Furthermore
\[
|\zeta|_{\tilde Y}^2 = 2 {\mathcal Q}(\zeta) = { -}3 {\mathcal S}(\zeta) +  \bigO(\varepsilon^{\frac{1}{2}}|\zeta|_{\tilde Y}^2) \lesssim |\zeta|_{\tilde Y}^3 + \varepsilon^{\frac{1}{2}} |\zeta|_{\tilde Y}^2,
\]
{ where we have used \eref{eq:dT} and the embedding $\tilde Y \hookrightarrow L^3(\R^2)$;
it follows that $|\zeta|_{\tilde Y} \gtrsim 1$.}
\qed

\begin{remark} \label{rem:inf is positive}
Let $c_\varepsilon := \inf_{N_\varepsilon} {\mathcal T}_\varepsilon$.
It follows from Proposition \ref{prop:lower bounds} that
$\liminf_{\varepsilon \to 0} c_\varepsilon { \gtrsim 1}$ and from
equation \eref{eq:only S} that
$-S(\zeta) \gtrsim c_\varepsilon -O(\varepsilon^{\frac{1}{2}})$ for all $\zeta \in N_\varepsilon$.
\end{remark}

The next result shows how points on $N_0$ may be approximated by points on $N_\varepsilon$.

\begin{proposition} \label{prop:approximate N0}
Suppose that ${\mathcal S}(\zeta_0) { <0}$ and $\lambda_0 \zeta_0 \in B_{M-1}(0)$
is the unique point on the ray through $\zeta_0 \in \tilde{Y}\setminus\{0\}$
which lies on $N_0$.
There exists $\xi_\varepsilon \in N_\varepsilon$ such that
$\lim_{\varepsilon \to 0} |\xi_\varepsilon -\lambda_0\zeta_0|_{\tilde{Y}} = 0$.
\end{proposition}
\proof Note that
\begin{equation}
\frac{\diff }{\diff  \lambda} {\mathcal T}_0(\lambda \zeta_0) \Big|_{\lambda=\lambda_0} = 0,
\qquad
\frac{\diff ^2}{\diff  \lambda^2} {\mathcal T}_0(\lambda \zeta_0) \Big|_{\lambda=\lambda_0} < 0.
\label{eq:zero ray}
\end{equation}
Let $\zeta_\varepsilon = \chi_\varepsilon(D)\zeta_0$, so that
$\zeta_\varepsilon\in \tilde{Y}_\varepsilon \subset \tilde{Y}$ with
$\lim_{\varepsilon\to 0}|\zeta_\varepsilon-\zeta_0|_{\tilde Y}=0$, and in particular
\[|\lambda_0\zeta_\varepsilon|_{\tilde Y} < M-1.\]
According to \eref{eq:zero ray} we can find $\tilde{\gamma}>1$ such that $\tilde{\gamma} |\lambda_0 \zeta_\varepsilon |_{\tilde Y} < M$
(so that $\tilde{\gamma}\lambda_0\zeta_\varepsilon \in U_\varepsilon$) and
\[
\frac{\diff }{\diff  \lambda} {\mathcal T}_0(\lambda \zeta_0) \Big|_{\lambda=\tilde{\gamma}^{-1}\lambda_0}>0,
\qquad
\frac{\diff }{\diff  \lambda} {\mathcal T}_0(\lambda \zeta_0) \Big|_{\lambda=\tilde{\gamma}\lambda_0}<0,
\]
and therefore
\[
\frac{\diff }{\diff  \lambda} {\mathcal T}_\varepsilon(\lambda \zeta_\varepsilon) \Big|_{\lambda=\tilde{\gamma}^{-1}\lambda_0}>0,
\qquad
\frac{\diff }{\diff  \lambda} {\mathcal T}_\varepsilon(\lambda \zeta_\varepsilon) \Big|_{\lambda=\tilde{\gamma}\lambda_0}<0
\]
(the quantities on the left-hand sides of the inequalities on the
second line converge to those on the first as $\varepsilon \to 0$).
It follows that there exists $\lambda_\varepsilon\in (\widetilde \gamma^{-1}\lambda_0,\widetilde \gamma \lambda_0)$
with
\[
\frac{\diff }{\diff  \lambda} {\mathcal T}_\varepsilon(\lambda \zeta_\varepsilon) \Big|_{\lambda=\lambda_\varepsilon}=0,
\]
that is, $\xi_\varepsilon:=\lambda_\varepsilon \zeta_\varepsilon \in N_\varepsilon$, and we conclude that
this value of $\lambda_\varepsilon$ is unique (see Proposition \ref{prop:nc interpretation})
and that $\lim_{\varepsilon \to 0} \lambda_\varepsilon = \lambda_0$.\qed

\begin{corollary}
Any minimising sequence $\{\zeta_n\}$ of ${\mathcal T}_\varepsilon|_{N_{\varepsilon}}$ satisfies
\[
\limsup_{n\to\infty} |\zeta_n|_{\tilde Y} <M-1.
\]
\end{corollary}
\proof In view of Proposition \ref{prop:lower bounds}
it sufficies to show that for each sufficiently large value of $M$
(chosen independently of $\varepsilon$) there exists
$\zeta^\star \in N_\varepsilon$ such that
${\mathcal T}_\varepsilon(\zeta^\star) < \frac{1}{12}(M-1)^2$. { In fact,} choose $\zeta_0\in \tilde{Y}\setminus\{0\}$ and $M>1$ such that
\begin{equation}
{\mathcal S}(\zeta_0) { < 0}, \qquad \frac{{\mathcal Q}(\zeta_0)^3}{{\mathcal S}(\zeta_0)^2}<\tfrac{27}{48}(M-1)^2.
\label{eq:choose M}
\end{equation}
The calculation
\[
\diff {\mathcal T}_0[\lambda_0 \zeta_0](\lambda_0\zeta_0)
= 2\lambda_0^2 {\mathcal Q}(\zeta_0) { +} 3\lambda_0^3 {\mathcal S}(\zeta_0)
\]
{ then} shows that $\lambda_0\zeta_0 \in N_0$, where
\[
\lambda_0={ -}\frac{2{\mathcal Q}(\zeta_0)}{3{\mathcal S}(\zeta_0)}.
\]
It follows that $\lambda_0\zeta_0$ is the unique point on its ray which lies on $N_0$, and
\begin{equation}
{\mathcal T}_0(\lambda_0\zeta_0)=\tfrac{1}{3}{\mathcal Q}(\lambda_0\zeta_0)=\frac{4{\mathcal Q}(\zeta_0)^3}{27{\mathcal S}(\zeta_0)^2}<\tfrac{1}{12}(M-1)^2, \label{eq:value of J0}
\end{equation}
so that
\[|\lambda_0 \zeta_0|_{\tilde Y} < M-1.\]
Proposition \ref{prop:approximate N0} asserts the existence of $\xi_\varepsilon \in N_\varepsilon$ with
$\lim_{\varepsilon \to 0}|\xi_\varepsilon - \lambda_0\zeta_0|_{\tilde{Y}}=0.$
Using the limit
\[\lim_{\varepsilon\to 0}
{\mathcal T}_\varepsilon(\xi_\varepsilon)
={\mathcal T}_0(\lambda_0\zeta_0)\]
and \eref{eq:value of J0}, we find that
\[{\mathcal T}_\varepsilon(\xi_\varepsilon)<\tfrac{1}{12}(M-1)^2.
\hspace{2.8in}\Box\]

The next step is to show that there is a minimising sequence for ${\mathcal T}_\varepsilon|_{N_\varepsilon}$
which is also a { Palais--Smale} sequence.

\begin{proposition}\label{prop:minimising sequence}
There exists a minimising sequence $\{\zeta_n\} \subset  B_{M-1}(0)$
of ${\mathcal T}_\varepsilon|_{N_\varepsilon}$ such that 
\[
\lim_{n \to \infty}|\diff {\mathcal T}_\varepsilon[\zeta_n]|_{\tilde Y_\varepsilon \to \R} = 0.
\]
\end{proposition}
\proof
Ekeland's variational principle for optimisation problems with regular constraints
(Ekeland \cite[Thm 3.1]{Ekeland74}) implies the existence of
a minimising sequence $\{\zeta_n\}$ for ${\mathcal T}_\varepsilon|_{N_\varepsilon}$ and a sequence
$\{\mu_n\}$ of real numbers such that
\[
\lim_{n \to \infty} |\diff {\mathcal T}_\varepsilon[\zeta_n] - \mu_n \, \diff {\mathcal G}_\varepsilon[\zeta_n]|_{\tilde Y_\varepsilon \to \R}  = 0.
\]
Applying this sequence of operators to $\zeta_n$, we find that $\mu_n \to 0$ as $n \to \infty$
(since $\diff {\mathcal T}_\varepsilon[\zeta_n](\zeta_n)=0$ and
$\diff{\mathcal G}_\varepsilon[\zeta_n](\zeta_n)=\diff ^2{\mathcal T}_\varepsilon[\zeta_n](\zeta_n,\zeta_n)\lesssim -1$),
whence $|\diff {\mathcal T}|_\varepsilon[\zeta_n]_{\tilde Y_\varepsilon \to \R} \to 0$ as $n \to \infty$.
\qed

The following lemma examines the convergence properties of more general { Palais--Smale} sequences.

\begin{lemma}\label{lemma:critical points} \hspace{2cm}
\begin{itemize}
\item[(i)]
Suppose that $\{\zeta_n\} \subset B_{M-1}(0)$ satisfies
\[\lim_{n \to \infty} \diff {\mathcal T}_\varepsilon[\zeta_n]=0,
\qquad
\sup_{j \in \Z^2} |\zeta_n|_{L^2(Q_j)} \gtrsim 1.\]
There exists $\{w_n\} \subset \Z^2$ with the property that a subsequence of
$\{\zeta_n(\cdot+w_n)\}$
converges weakly in $\tilde{Y}_\varepsilon$ to a nontrivial critical point $\zeta_\infty$ of
${\mathcal T}_\varepsilon$.
\item[(ii)]
Suppose that $\varepsilon>0$. The corresponding sequence of
FDKP-solutions $\{u_n\}$, where
\[
u_n = u_1(\tilde{u}_1(\zeta_n)) + u_2(u_1(\tilde{u}_1(\zeta_n)))
\]
and we have abbreviated $\{\zeta_n(\cdot+w_n)\}$ to $\{\zeta_n\}$, converges weakly in
$X$  to $u_\infty = u_1(\tilde{u}_1(\zeta_\infty)) + u_2(u_1(\tilde{u}_1(\zeta_\infty)))$
(which is a nontrivial critical point of ${\mathcal I}_\varepsilon$).
\end{itemize}
\end{lemma}
\proof
We can select $\{w_n\} \subset \Z^2$ so that
\[
\liminf_{n \to \infty} |\zeta_n(\cdot+w_n)|_{L^2({Q_0})} { \gtrsim 1}.
\]
The sequence $\{\zeta_n(\cdot+w_n)\} \subset B_{M-1}(0)$
admits a subsequence which converges weakly in $\tilde{Y}_\varepsilon$
and strongly in $L^2({Q_0})$ to $\zeta_\infty \in B_M(0)$;
it follows that $|\zeta_\infty|_{L^2({Q_0})} > 0$ and therefore $\zeta_\infty \neq 0$.
We henceforth abbreviate $\{\zeta_n(\cdot+w_n)\}$ to $\{\zeta_n\}$ and extract further subsequences
as necessary.

We first treat the case $\varepsilon=0$. For $w \in C_0^\infty(\R^2)$ we find that
\[
\int_{\R^2} (\zeta_n^2 - \zeta_\infty^2) w \dx \dy \leq |\zeta_n-\zeta_\infty|_{L^3(|(x,y)|<R)}^2
|w|_{L^3} \to 0
\]
as $n \to \infty$, where $R$ is chosen so that $\supp w \subset \{|(x,y)|<R\}$
 ($\{\zeta_n\}$ converges
strongly to $\zeta_\infty$ in $ L^3(|(x,y)|<R)$. This result also holds for $w \in L^3(\R^2)$ (by density) and hence for all $w \in \tilde{Y}$
(because $\tilde{Y} \subset L^3(\R^2))$. Furthermore $\langle \zeta_n, w \rangle_{\tilde Y} \to \langle \zeta_\infty,w \rangle_{\tilde Y}$
as $n \rightarrow \infty$ for all $w \in \tilde{Y}$. By taking the limit $n \rightarrow \infty$ in the equation
\[
\diff {\mathcal T}_0[\zeta_n](w)=\langle \zeta_n, w \rangle_{\tilde Y} + \int_{\R^2} \zeta_n^2 w \dx \dy,
\]
one therefore finds that
\[\langle \zeta_\infty, w \rangle_{\tilde Y} + \int_{|(x,y)| < R} \zeta_\infty^2 w \dx \dy=0,\]
that is, $\diff {\mathcal T}_0[\zeta_\infty](w)=0$ for all $w \in \tilde{Y}$.
It follows that $\diff {\mathcal T}_0[\zeta_\infty]=0$.

Now suppose that $\varepsilon>0$. According to Proposition \ref{prop:weak trace back}
the sequence $\{u_n\}$ converges weakly in $X$ to $u_\infty$, and the remarks below
equation \eref{eq:I vs T} show that
\[
\lim_{n \to \infty}|\diff {\mathcal I}_\varepsilon[u_n]|_{X \to \R} =0.
\]
Since $u \mapsto \varepsilon u + n(D)u + u^2$ is weakly continuous
$X \mapsto L^2(\R^2)$ (see Proposition \ref{prop:FDKP op weak} and
Lemma \ref{lemma:global embeddings}(i)), one finds that
\begin{eqnarray*}
\diff {\mathcal I}_\varepsilon[u_\infty](w) &= & \int_{\R^2} \left( \varepsilon^2 u_\infty + n(\Diff) u_\infty + u_\infty^2 \right) w \dx \dy\nonumber \\
&=&  \lim_{n \to \infty} \int_{\R^2} \left( \varepsilon^2 u_n + n(\Diff) u_n + u_n^2 \right) w \dx \dy \nonumber \\
& = & \lim_{n \to \infty}\diff {\mathcal I}_\varepsilon[u_n](w) \nonumber \\
& = & 0
\end{eqnarray*}
for any $w \in \partial_x { \Schwartz(\R^2)}$, whence $u_\infty$ is a critical point of ${\mathcal I}_\varepsilon$
(so that $\zeta_\infty$ is a critical point of ${\mathcal T}_\varepsilon$).\qed

It remains to show that the minimising sequence
for ${\mathcal T}_\varepsilon$ over $N_\varepsilon$ identified in Proposition
\ref{prop:minimising sequence} satisfies the `nonvanishing' criterion in Lemma
\ref{lemma:critical points}. This task is accomplished in Proposition \ref{prop:17/6}
and Corollary \ref{cor:no vanishing} below.

\begin{proposition} \label{prop:17/6}
The inequality
\[\int_{\R^2} |\zeta| |\xi|^2 \dx\dy \lesssim  \sup_{j \in \Z^2} |\zeta|_{L^2(Q_j)}^\frac{1}{6}
 |\zeta|_{\tilde{Y}}^{\frac{5}{6}} |\xi|_{\tilde Y}^2
 \]
{ holds for all $\zeta, \xi \in \tilde{Y}$.}
 \end{proposition}
 
\proof This result follows from the calculation
\begin{eqnarray*}
\int_{\R^2} |\zeta| |\xi|^2 \dx\dy
 & \lesssim &  |\zeta |_{L^3} |\zeta|_{L^3}^2 \\
& \lesssim & 
  \left( \sum_{j \in \Z^2}  |\zeta|_{L^3(Q_j)}^3\right)^{\!\!\!\frac{1}{3}} 
 |\xi|_{\tilde{Y}}^2 \\
 & \lesssim & \left(\sup_{j \in \Z^2} |\zeta|_{L^3(Q_j)}
 \sum_{j \in \Z^2}  |\zeta|_{L^3(Q_j)}^2\right)^{\!\!\!\frac{1}{3}} 
 |\xi|_{\tilde{Y}}^2 \\ 
  & \lesssim & \left(\sup_{j \in \Z^2} |\zeta|_{L^2(Q_j)}^\frac{1}{2}|\zeta|_{\tilde{Y}(Q_j)}^\frac{1}{2}
 \sum_{j \in \Z^2}  |\zeta|_{\tilde{Y}(Q_j)}^2\right)^{\!\!\!\frac{1}{3}} 
 |\xi|_{\tilde{Y}}^2 \\ 
 & \lesssim & \sup_{j \in \Z^2} |\zeta|_{L^2(Q_j)}^\frac{1}{6}
 |\zeta|_{\tilde{Y}}^{\frac{5}{6}} |\xi|_{\tilde{Y}}^2,
\end{eqnarray*}
where we have interpolated between $L^2(Q_j)$ and $L^6(Q_j)$
and used the embeddings $L^3({\mathbb R}^2) \hookrightarrow \tilde{Y}$,
$L^6(Q_j) \hookrightarrow \tilde{Y}(Q_j)$ and
 $\ell^\infty({\mathbb Z}^2, \tilde{Y}(Q_j))
\hookrightarrow \ell^2({\mathbb Z}^2, \tilde{Y}(Q_j))=\tilde{Y}$.\qed

\begin{corollary} \label{cor:no vanishing}
Any sequence $\{\zeta_n\} \subset  N_\varepsilon$
satisfies
\[
\sup_{j \in \Z^2}  |\zeta_n|_{L^2(Q_j)} \gtrsim 1.
\]
\end{corollary}
\proof
Using Proposition \ref{prop:17/6}, one finds that
\[
{ |{\mathcal S}(\zeta_n)|} \leq \int_{\R^2} |\zeta_n||\zeta_n|^2\dx\dy \lesssim
\sup_{j \in \Z^2} |\zeta|_{L^2(Q_j)}^\frac{1}{6} |\zeta|_{\tilde{Y}}^{\frac{17}{6}} \lesssim
\sup_{j \in \Z^2} |\zeta|_{L^2(Q_j)}^\frac{1}{6}
\]
(because $|\zeta_n|_{\tilde{Y}} < M$),
 and the result follows from this estimate and the fact that
${-}S(\zeta_n) \geq c_\varepsilon-\bigO(\varepsilon^{\frac{1}{2}})$ with $\liminf_{\varepsilon \to 0} c_\varepsilon { \gtrsim 1}$ { (see Remark~\ref{rem:inf is positive}).}\qed

\begin{theorem}\label{thm:first existence theorem} \hspace{2cm}
\begin{itemize}
\item[(i)]
Let $\{\zeta_n\} \subset B_{M-1}(0)$ be a minimising sequence for ${\mathcal T}_\varepsilon|_{N_\varepsilon}$
with
\[
\lim_{n \to \infty} |\diff {\mathcal T}_\varepsilon[\zeta_n]|_{\tilde Y_\varepsilon \to \R} =0.
\] 
There exists $\{w_n\} \subset \Z^2$ such that a subsequence of $\{\zeta_n(\cdot+w_n)\}$ converges weakly
in $\tilde{Y}_\varepsilon$ to a nontrivial critical point $\zeta_\infty$ of ${\mathcal T}_\varepsilon$.
\item[(ii)]
Suppose that $\varepsilon>0$. The corresponding sequence of
FDKP-solutions $\{u_n\}$, where
\[
u_n = u_1(\tilde{u}_1(\zeta_n)) + u_2(u_1(\tilde{u}_1(\zeta_n)))
\]
and we have abbreviated $\{\zeta_n(\cdot+w_n)\}$ to $\{\zeta_n\}$, converges weakly in
$X$  to $u_\infty = u_1(\tilde{u}_1(\zeta_\infty)) + u_2(u_1(\tilde{u}_1(\zeta_\infty)))$
(which is a nontrivial critical point of ${\mathcal I}_\varepsilon$).
\end{itemize}
\end{theorem}

\section{Ground states} \label{sec:ground states}

In this section we improve the result of Theorem \ref{thm:first existence theorem} by showing that we
can choose the sequence $\{w_n\}$ to ensure convergence to a ground state. 
For this purpose we use the following abstract concentration-compactness theorem,
which is a straightforward modification of theory
given by Buffoni, Groves \& Wahl\'{e}n \cite[Appendix A]{BuffoniGrovesWahlen18}.

\begin{theorem} \label{thm:cc}
Let $H_0$, $H_1$ be Hilbert spaces and $H_1$ be continuously embedded in $H_0$.
Consider a sequence $\{x_n\}$ in $\ell^2(\Z^s,H_1)$, where $s \in \N$.
Writing $x_n=(x_{n,j})_{j\in \Z^s}$, where  $x_{n,j}\in H_1$, suppose that
\begin{itemize}
\item[(i)] $\{x_n\}$ is bounded in $\ell^2(\Z^s,H_1)$,
\item[(ii)] $S=\{x_{n,j}:n \in \N, j \in \Z^s\}$ is relatively compact in $H_0$,
\item[(iii)] $\limsup_{n\to \infty}|x_n|_{\ell^\infty(\Z^s,H_0)} { \gtrsim 1}$.
\end{itemize}

For each $\Delta>0$ the sequence $\{x_n\}$ admits a subsequence
with the following properties.
{ There exist a finite number $m$ of non-zero vectors $x^1,\ldots,x^m\in \ell^2(\Z^s,H_1)$ and
sequences $\{w^1_n\}$, \ldots, $\{w^m_n\}
 \subset \Z^s$ satisfying
\[
\lim_{n \to \infty} |w_n^{m^{\prime\prime}}-w_n^{m^\prime}| \to \infty, \qquad 1 \leq m^{\prime\prime} < m^\prime \leq m
\]
such that
\begin{eqnarray*}
& &  T_{-w^{m^\prime}_n}x_n\rightharpoonup x^{m^\prime}, \\
& &
 |x^{m^\prime}|_{l^\infty(\Z^s,H_0)}=
\lim_{n\to \infty}
\left|x_n-\sum_{\ell=1}^{m^\prime-1}T_{w^\ell_n}x^\ell\right|_{l^\infty(\Z^s,H_0)}, \\
& &
 \lim_{n\to \infty}|x_n|_{\ell^2(\Z^s,H_1)}^2=
\sum_{\ell=1}^{m^\prime} |x^\ell|_{\ell^2(\Z^s,H_1)}^2+
\lim_{n\to \infty}
\left|x_n-\sum_{\ell=1}^{m^\prime}T_{w^\ell_n}x^\ell\right|_{\ell^2(\Z^s,H_1)}^2
\end{eqnarray*}
for $m^\prime = 1, \ldots, m$,
\[
\limsup_{n\to \infty}\left\|
x_n-\sum_{\ell=1}^mT_{w^\ell_n}x^\ell
\right\|_{\ell^\infty(\Z^s,H_0)}\leq \Delta,
\]
and
\[
\lim_{n\to \infty}\left\|x_n-T_{w^1_n}x^1
\right\|_{\ell^\infty(\Z^s,H_0)}=0 
\]
if $m=1$. 
Here the weak convergence is understood in $\ell^2(\Z^s,H_1)$ and
$T_w$ denotes the translation operator $T_w (x_{n,j})=(x_{n,j-w})$.}
\end{theorem}

We proceed by using Theorem \ref{thm:cc} to study { Palais--Smale} sequences for
${\mathcal T}_\varepsilon$, extracting subsequences where necessary for the
validity of our arguments.

\begin{lemma} \label{lemma:application of cc}
Suppose that $\{\zeta_n\} \subset B_{M-1}(0)$ satisfies
\[\lim_{n \to \infty} \diff {\mathcal T}_\varepsilon[\zeta_n]=0,
\qquad
\sup_{j \in \Z^2} |\zeta_n|_{L^2(Q_j)} \gtrsim 1.\]
There exists $\{w_n\} \subset \Z^2$ and $\zeta_\infty$ such that $\zeta_n(\cdot+w_n)
\rightharpoonup \zeta_\infty$ in $\tilde{Y}$, ${\mathcal S}(\zeta_n) \to
{\mathcal S}(\zeta_\infty)$ as $n \to \infty$ and
\[\lim_{n \to \infty} \sup_{j \in \Z^2} |\zeta_n(\cdot+w_n)-\zeta_\infty|_{L^2(Q_j)}=0.\]
\end{lemma}
\proof Set
$H_1=\tilde{Y}({Q_0})$,
$H_0=L^2({Q_0})$,
define $x_n\in \ell^2(\Z^2,{ H_1})$ for $n \in \N$ by
\[
x_{n,j}=\zeta_n(\cdot+j)|_{{Q_0}}
\in \tilde{Y}({Q_0}), \qquad j\in \Z^2,
\]
and apply Theorem \ref{thm:cc} to the sequence $\{x_n\}\subset \ell^2(\Z^2,H_1)$,
noting that
\[
|x_n|_{\ell^2(\Z^2,H_1)}=|\zeta_n|_{\tilde{Y}}, \qquad
|x_n|_{\ell^\infty(\Z^2,H_0)}=\sup_{j\in\Z^2}|
\zeta_n
|_{L^2(Q_j)}
\]
for $n \in \N$. Assumption (ii) is satisfied because $\tilde{Y}$ is compactly embedded
in $L^2({Q_0})$, while assumptions (i) and (iii) follow from the hypotheses in the lemma.

The theorem
asserts the existence of a natural number $m$, sequences
$\{w_n^1\}, \ldots, \{w_n^m\} \subset \Z^2$ with
\[
\lim_{n \to \infty} |w_n^{m^{\prime\prime}}-w_n^{m^\prime}| { =} \infty, \qquad 1 \leq m^{\prime\prime} < m^\prime \leq m,
\]
and functions
$\zeta^1,\ldots,\zeta^m\in B_M(0)\setminus\{0\}$ such that
$\zeta_n(\cdot+w^{m^\prime}_n) \rightharpoonup \zeta^{m^\prime}$ in $\tilde{Y}$ as $n \to \infty$,
\[
\limsup_{n\to \infty}\sup_{j\in\Z^2}\left\|
\zeta_n-\sum_{\ell=1}^m\zeta^\ell(\cdot-w^\ell_n)
\right\|_{L^2(Q_j)}
\leq \varepsilon^6,
\]
\[
\sum_{\ell=1}^m\|\zeta^\ell\|_{\tilde{Y}}^2\leq
\limsup_{n\to \infty}\|\zeta_n\|_{\tilde{Y}}^2
\]
and
\begin{equation}
\label{eq:concentrate}
\lim_{n\to \infty}\sup_{j\in\Z^2}\left\|\zeta_n-\zeta^1(\cdot-w^1_n)
\right\|_{L^2(Q_j)}
=0
\end{equation}
if $m=1$. It follows from Lemma \ref{lemma:critical points}(i)
that $\mathrm{d}{\mathcal T}_\varepsilon[\zeta^\ell]=0$, so that
$\zeta^\ell \in N_\varepsilon$ and ${\mathcal T}_\varepsilon(\zeta^\ell) \geq c_\varepsilon { \gtrsim 1}$.

Define
\[\tilde{\zeta}_n=\sum_{\ell=1}^m\zeta^\ell(\cdot-w^\ell_n), \qquad n \in \N,\]
and note that
\begin{equation}
{\mathcal S}(\tilde{\zeta}_n) \to \sum_{\ell=1}^m {\mathcal S}(\zeta^\ell)
\label{eq:first S estimate}
\end{equation}
as $n \to \infty$ (approximate $\zeta^\ell \in L^3(\R^2)$ by a sequence of functions in $C_0^\infty(\R^2)$ and use the fact that $|w_n^{\ell_1} - w_n^{\ell_2}| \to \infty$).
Furthermore, { from} Proposition \ref{prop:17/6}, one finds that
\begin{eqnarray}
\afl \lefteqn{\qquad\quad\limsup_{n \to \infty} |{\mathcal S}(\zeta_n) - {\mathcal S}(\tilde{\zeta}_n)|}\nonumber\\
 & \lesssim & \limsup_{n \to \infty}  \int_{\R^2} |\zeta_n-\tilde{\zeta}_n|(|\zeta_n|^2+|\tilde{\zeta}_n|^2) \dx\dy \nonumber \\
 & \lesssim & { \limsup_{n \to \infty} \sup_{j \in \Z^2}} |\zeta_n-\tilde{\zeta}_n|_{L^2(Q_j)}^\frac{1}{6} \limsup_{n \to \infty} 
 |\zeta_n-\tilde{\zeta}_n|_{\tilde{Y}}^{\frac{5}{6}} (|\zeta_n|_{\tilde{Y}}^2 + |\tilde{\zeta}_n|_{\tilde{Y}}^2 )  \label{eq:intermediate S estimate} \\
 & \leq & \varepsilon \limsup_{n \to \infty} (|\zeta_n|_{\tilde Y}^2+|\tilde{\zeta}_n|_{\tilde Y}^2)^{\frac{17}{12}}
\nonumber \\
 & \lesssim & \varepsilon \label{eq:second S estimate}
\end{eqnarray}
uniformly in $m$.
Combining \eref{eq:first S estimate}, \eref{eq:second S estimate} and
\[
{ -} {\mathcal S}(\zeta^\ell) \geq c_\varepsilon-\bigO(\varepsilon^{\frac{1}{2}})|\zeta^{\ell}|^2_{\tilde{Y}}, \qquad \ell=1,\ldots,m,
\]
yields
\[
-  \limsup_{n\to \infty} {\mathcal S}(\zeta_n)
\geq m c_\varepsilon-\bigO(\varepsilon^{\frac{1}{2}})
\]
and hence
\[
c_\varepsilon\geq mc_\varepsilon-\bigO(\varepsilon^{\frac{1}{2}})
\]
uniformly in $m$ (because of \eref{eq:only S}). It follows that $m=1$ (recall that $\liminf_{\varepsilon\to 0}c_\varepsilon { \gtrsim 1}$).

The advertised result now follows from
\eref{eq:concentrate} (with $\zeta_\infty=\zeta^1$ and $w_n=w_n^1$) and
\eref{eq:intermediate S estimate} (since
${\mathcal S}(\tilde{\zeta}_1) = {\mathcal S}(\zeta^1)$).\qed

We can now strengthen Theorem \ref{thm:first existence theorem}, dealing with the cases
$\varepsilon=0$ and $\varepsilon>0$ separately.

\begin{lemma} \label{lemma:general convergence}
Suppose that $\{\zeta_n\} \subset B_{M-1}(0)$ satisfies
\[
\lim|\diff {\mathcal T}_0[\zeta_n]|_{\tilde Y \to \R} = 0,
\qquad
\sup_{j \in \Z^2} |\zeta_n|_{L^2(Q_j)} \gtrsim 1.\]
There exists $\{w_n\} \subset \Z^2$ such that
$\{\zeta_n(\cdot+w_n)\}$ converges strongly in $\tilde{Y}$ to a nontrivial critical point of ${\mathcal T}_0$.
\end{lemma}
\proof Lemma \ref{lemma:application of cc} asserts the existence of 
$\{w_n\} \subset \Z^2$ and $\zeta_\infty \neq 0$ such that
$\zeta_n(\cdot+w_n)
\rightharpoonup \zeta_\infty$ in $\tilde{Y}$ and  ${\mathcal S}(\zeta_n) \to
{\mathcal S}(\zeta_\infty)$ as $n \to \infty$. Abbreviating $\{\zeta_n(\cdot+w_n)\}$
to $\{\zeta_n\}$, we find from \eref{eq:dT} that
\[{\mathcal Q}(\zeta_n) = \tfrac{1}{2}\diff {\mathcal T}_0[\zeta_n](\zeta_n)  { -} \tfrac{3}{2}{\mathcal S}(\zeta_n)
\to  { -}\tfrac{3}{2}{\mathcal S}(\zeta_\infty) = {\mathcal Q}(\zeta_\infty),\]
that is, $|\zeta_n|_{\tilde{Y}}^2 \to |\zeta_\infty|_{\tilde Y}^2$ as $n \to \infty$.
It follows that $\zeta_n \to \zeta_\infty$ in $\tilde{Y}$ as $n \to \infty$ and
in particular that $\diff {\mathcal T}_0[\zeta_\infty] = 0.$\qed

{ We obtain the following existence result in the case $\varepsilon = 0$ as a direct corollary of Lemma~\ref{lemma:general convergence}.}

\begin{theorem} \label{thm:second existence result, epsilon zero}
Let $\{\zeta_n\} \subset B_{M-1}(0)$ be a minimising sequence for ${\mathcal T}_0|_{N_0}$
with
\[
\lim|\diff {\mathcal T}_0[\zeta_n]|_{\tilde Y \to \R} = 0.
\] 
There exists $\{w_n\} \subset \Z^2$ such that $\{\zeta_n(\cdot+w_n)\}$
converges strongly in $\tilde{Y}$ to a ground state of ${\mathcal T}_0$.
\end{theorem}

Let us now turn to the case $\varepsilon>0$. We begin with the following observation.

\begin{remark} \label{rem:unbounded sequences}
Suppose that $u_n \rightharpoonup u_\infty$ in { $H^s(\R^2)$} as $n \to \infty$. The limit
\[
\lim_{n\to \infty} |u_n-u_\infty|_\infty = 0
\]
holds if and only if $u_n(\cdot - j_n) \rightharpoonup 0$ in $H^s(\R^2)$ as $n \to \infty$ for all
unbounded sequences $\{j_n\} \subset \Z^2$.
\end{remark}
\begin{theorem}
{ Let $\varepsilon > 0$ and} $\{\zeta_n\} \subset B_{M-1}(0)$ be a minimising sequence for ${\mathcal T}_\varepsilon|_{N_\varepsilon}$
with
\[
\lim_{n \to \infty}|\diff {\mathcal T}_\varepsilon[\zeta_n]|_{\tilde Y \to \R}=0.
\] 
There exists $\{w_n\} \subset \Z^2$ such that $\{\zeta_n(\cdot+w_n)\}$
converges weakly in $\tilde{Y}_\varepsilon$ to a ground state $\zeta_\infty$ of ${\mathcal T}_\varepsilon$.
The corresponding sequence of
FDKP-solutions $\{u_n\}$, where
\[
u_n = u_1(\tilde{u}_1(\zeta_n)) + u_2(u_1(\tilde{u}_1(\zeta_n)))
\]
and we have abbreviated $\{\zeta_n(\cdot+w_n)\}$ to $\{\zeta_n\}$, converges weakly in
$X$ and strongly in $L^\infty(\R^2)$ to $u_\infty = u_1(\tilde{u}_1(\zeta_\infty)) + u_2(u_1(\tilde{u}_1(\zeta_\infty)))$
(which is a nontrivial critical point of ${\mathcal I}_\varepsilon$).
\end{theorem}
\proof Lemma \ref{lemma:application of cc} asserts the existence of 
$\{w_n\} \subset \Z^2$ and $\zeta_\infty \neq 0$ such that
$\zeta_n(\cdot+w_n)
\rightharpoonup \zeta_\infty$ in $\tilde{Y}$ as $n \to \infty$ and
\[
\lim_{n \to \infty} \sup_{j \in \Z^2} |\zeta_n(\cdot+w_n)-\zeta_\infty|_{{ H^s(Q_j)}}=0,
\]
{ where we have estimated
\begin{eqnarray*}
\afl | \zeta_n(\cdot+w_n)-\zeta_\infty |_{H^s(Q_j)}^2 & \lesssim  | \zeta_n(\cdot+w_n)-\zeta_\infty |_{L^2(Q_j)} |\zeta_n(\cdot+w_n)-\zeta_\infty|_{H^{2s}(Q_j)}\\
&\lesssim | \zeta_n(\cdot+w_n)-\zeta_\infty |_{L^2(Q_j)}  |\zeta_n(\cdot+w_n)-\zeta_\infty|_{H^{2s}(\R^2)}\\
&\lesssim |\zeta_n(\cdot+w_n)-\zeta_\infty|_{L^2(Q_j)}
\end{eqnarray*}
because $\{\zeta_n(\cdot+w_n)-\zeta_\infty\}$ is bounded in $\tilde Y_\varepsilon$
(which coincides with $H_\varepsilon^{2s}(\R^2)$).
It follows that
\[
 \lim_{n \to \infty} \sup_{j \in \Z^2} |\zeta_n-\zeta_\infty|_{L^\infty(Q_j)} = \lim_{n \to \infty} |\zeta_n-\zeta_\infty|_\infty =0,
\]
where have again abbreviated $\{\zeta_n(\cdot+w_n)\}$ to $\{\zeta_n\}$,
and Remark \ref{rem:unbounded sequences} shows that
$\zeta_n(\cdot - j_n) \rightharpoonup 0$ in $ H^s_\varepsilon(\R^2)$
and hence in $\tilde{Y}_\varepsilon$ as $n \to \infty$ for all unbounded sequences $\{j_n\} \subset \Z^2$.

Using Proposition \ref{prop:weak trace back}, one finds that $u_n(\cdot - j_n) \rightharpoonup 0$ in $X$ and hence in $ H^s(\R^2)$
for all unbounded sequences $\{j_n\} \subset \Z^2$, so that $u_n \to u_\infty$ in
$L^\infty(\R^2)$ as $n \to \infty$ (Remark \ref{rem:unbounded sequences}). It follows that $u_n \to u_\infty$ in
$L^3(\R^2)$ and in particular that ${\mathcal S}(u_n) \to {\mathcal S}(u_\infty)$ as $n \to \infty$. Since
$\diff {\mathcal I}_\varepsilon[u_\infty]=0$ and $\diff {\mathcal I}_\varepsilon[u_n](u_n) \to 0$ as $n \to \infty$
(see the remarks below equation \eref{eq:I vs T}), one finds from the identity
\[{\mathcal I}_\varepsilon(u)=\tfrac{1}{2}\diff {\mathcal I}_\varepsilon[u](u) { - \tfrac{1}{2}} {\mathcal S}(u)\]
that ${\mathcal T}_\varepsilon(\zeta_n)
\to {\mathcal T}_\varepsilon(\zeta_\infty)$ as $n \to \infty$, so that ${\mathcal T}_\varepsilon(\zeta_\infty)=c_\varepsilon$.\qed

Finally, we show that critical points of ${\mathcal T}_\varepsilon$ converge to critical points of ${\mathcal T}_0$
as $\varepsilon \to 0$. The first step is to establish the corresponding convergence result for the infima of
these functionals over their natural constraint sets.

\begin{lemma} \label{lem:infima converge}
One has that $\lim_{\varepsilon \to 0} c_\varepsilon = c_0$.
\end{lemma}
\proof
Let $\{\varepsilon_n\}$ be a sequence with $\lim_{n \to \infty} \varepsilon_n=0$ and
$\zeta^{\varepsilon_n}$, $\zeta^0$ be a ground states of respectively ${\mathcal T}_{\varepsilon_n}$
and ${\mathcal T}_0$.

Because $\varepsilon^\frac{1}{2}{\mathcal R}_\varepsilon\circ\chi_\varepsilon(D)$ and $\varepsilon^\frac{1}{2}\diff {\mathcal R}_\varepsilon\circ\chi_\varepsilon(D)$
converge uniformly to zero over $B_{M-1}(0) \subset \tilde{Y}$ as $\varepsilon \to 0$,
 we find that
\[{\mathcal T}_{\varepsilon_n}(\zeta^{\varepsilon_n})-{\mathcal T}_0(\zeta^{\varepsilon_n})=o(1),
\qquad
\diff{\mathcal T}_{\varepsilon_n}[\zeta^{\varepsilon_n}]-\diff{\mathcal T}_0[\zeta^{\varepsilon_n}]=o(1)
\]
as $n \to \infty$ and hence that
\[\lim_{n \to \infty}|\diff{\mathcal T}_0[\zeta^{\varepsilon_n}]|_{\tilde{Y} \to \R} =0.\]
Proposition \ref{prop:17/6} implies that
\[{\mathcal S}(\zeta^{\varepsilon_n}) \leq \int_{\R^2} |\zeta^{\varepsilon_n}||\zeta^{\varepsilon_n}|^2\dx\dy \lesssim
\sup_{j \in \Z^2} |\zeta|_{L^2(Q_j)}^\frac{1}{6} |\zeta|_{\tilde{Y}}^{\frac{17}{6}} \lesssim
\sup_{j \in \Z^2} |\zeta|_{L^2(Q_j)}^\frac{1}{6}\]
(because $|\zeta^{\varepsilon_n}|_{\tilde{Y}} < M$),
and combining this estimate with
$S(\zeta^{\varepsilon_n}) \geq c_\varepsilon-\bigO(\varepsilon^{\frac{1}{2}})$ and $\liminf_{\varepsilon \to 0} c_\varepsilon { \gtrsim 1}$
yields
\[\sup_{j \in \Z^2} |\zeta^{\varepsilon_n}|_{L^2(Q_j)} \gtrsim 1.\]
According to Lemma \ref{lemma:general convergence} there exists
$\{w_n\} \subset \Z^2$ and $\zeta^\star \in N_0$ such that $\diff{\mathcal T}_0[\zeta^\star]=0$ and
$ \zeta^{\varepsilon_n}(\cdot+w_n) \to \zeta^\star$ in $\tilde{Y}$ as $n \to \infty$.
It follows that
\begin{eqnarray}
c_0 & \leq & {\mathcal T}_0(\zeta^\star) \nonumber \\
& = & \lim_{n \to \infty} {\mathcal T}_0(\zeta^{\varepsilon_n}) \nonumber \\
& = & \lim_{n \to \infty} 
\big({\mathcal T}_0(\zeta^{\varepsilon_n})-{\mathcal T}_{\varepsilon_n}(\zeta^{\varepsilon_n})\big)
+ \lim_{n \to \infty} 
\big({\mathcal T}_{\varepsilon_n}(\zeta^{\varepsilon_n})-c_{\varepsilon_n}\big)
+ \liminf_{n \to \infty} c_{\varepsilon_n} \nonumber \\
& = &  \liminf_{n \to \infty} c_{\varepsilon_n}. \label{eq:inf est}
\end{eqnarray}

Proposition \ref{prop:approximate N0} (with $\lambda_0=1$ and $\zeta_0=\zeta^0$)
asserts the existence of $\xi_n \in N_{\varepsilon_n}$
with $\xi_n \to \zeta^0$ in $\tilde{Y}$ and hence ${\mathcal T}_0(\xi_n) \to
{\mathcal T}_0(\zeta^0)=c_0$ as $n \to \infty$.
Because $\varepsilon^\frac{1}{2}{\mathcal R}_\varepsilon\circ \chi_\varepsilon(D)$
converges uniformly to zero over $B_{M-1}(0) \subseteq \tilde{Y}$ as
$\varepsilon \to 0$, one finds that
\[{\mathcal T}_\varepsilon(\xi_n)-{\mathcal T}_0(\xi_n) = o(1)\]
as $n \to \infty$, whence
\begin{eqnarray}
\limsup_{n \to \infty} c_{\varepsilon_n}
& \leq & \limsup_{n \to \infty} {\mathcal T}_{\varepsilon_n}(\xi_n) \nonumber \\
& = & \lim_{n \to \infty}\big({\mathcal T}_{\varepsilon_n}(\xi_n) - {\mathcal T}_0(\xi_n)\big)
+ \lim_{n \to \infty}\big({\mathcal T}_0(\xi_n) - c_0\big) + c_0 \nonumber \\
& = & c_0. \label{eq:sup est}
\end{eqnarray}
The stated result follows from inequalities \eref{eq:inf est} and \eref{eq:sup est}.\qed

\begin{corollary}
Let $\{\varepsilon_n\}$ be a sequence with $\lim_{n \to \infty} \varepsilon_n=0$ and
$\zeta^{\varepsilon_n}$ be a ground state of ${\mathcal T}_{\varepsilon_n}$.
There exists
$\{w_n\} \subset \Z^2$ and a ground state $\zeta^\star$ of ${\mathcal T}_0$ such that { a subsequence of
$\{\zeta^{\varepsilon_n}(\cdot+w_n)\}_n$ converges to  $\zeta^\star$} in $\tilde{Y}$ as $n \to \infty$.
\end{corollary}
\proof Continuing the arguments in the proof of Lemma \ref{lem:infima converge}, we find that it remains only
to show that ${\mathcal T}_0(\zeta^\star)=c_0$. This fact follows from the { calculations}
\[{\mathcal T}_{\varepsilon_n}(\zeta^{\varepsilon_n})-{\mathcal T}_0(\zeta^{\varepsilon_n})=o(1), \qquad {\mathcal T}_{\varepsilon_n}(\zeta^{\varepsilon_n}) = c_{\varepsilon_n} \to c_0\]
 and
$\zeta^{\varepsilon_n}(\cdot+w_n) \to \zeta^\star$ in $\tilde{Y}$ as $n \to \infty$.\qed

Finally, we record the corresponding result for FDKP solutions.

\begin{theorem} 
Let $\{\varepsilon_n\}$ be a sequence with $\lim_{n \to \infty} \varepsilon_n=0$
and $u^{\varepsilon_n}$ be a critical point of ${\mathcal I}_{\varepsilon_n}$
with ${\mathcal I}_{\varepsilon_n}(u^{\varepsilon_n})=\varepsilon_n^3c_{\varepsilon_n}$,
so that the formula $u^{\varepsilon_n}=u_1(\zeta^{\varepsilon_n}) + u_2(u_1(\zeta^{\varepsilon_n}))$
defines a ground state $\zeta^{\varepsilon_n}$ of ${\mathcal T}_\varepsilon$. There exists
$\{w_n\} \subset \Z^2$ and a ground state $\zeta^\star$ of ${\mathcal T}_0$ such that { a subsequence of
$\{\zeta^{\varepsilon_n}(\cdot+w_n)\}$ converges to} $\zeta^\star$ in $\tilde{Y}$ as $n \to \infty$.
\end{theorem}
\begin{remark}\label{remark:u-convergence}
Define
$u^\star_{\varepsilon}(x,y) = \varepsilon^2 \zeta^\star(\varepsilon x,\varepsilon^2 y)$,
so that $u^\star_\varepsilon$ is a KP solitary wave with wave speed $\varepsilon^2$ (see
the comments above equation \eref{eq:normalised steady KP}).
Abbreviating $u_1(\zeta^{\varepsilon}(\cdot+w_n))$,
$u_2(u_1(\zeta^{\varepsilon}(\cdot+w_n)))$ to $u_1^\varepsilon$, $u_2^\varepsilon$,
one finds that
the convergence $|\zeta^\star - \zeta^{\varepsilon_n}(\cdot+w_{\varepsilon_n})|_{\tilde{Y}} = o(1)$
translates to $| u^\star_{\varepsilon_n} - u^{\varepsilon_n}_1|_{\varepsilon_n} = o(\varepsilon_n^{\frac{1}{2}})$, and by Lemma~\ref{lemma:u_2}, $|u^{\varepsilon_n}_2|_{\varepsilon_n} \lesssim \varepsilon_n |u^{\varepsilon_n}_1|_{\varepsilon_n}^2 \lesssim \varepsilon_n^3$ is negligible in comparison. It follows that
$| u^\star_{\varepsilon_n} - u^{\varepsilon_n}|_{\varepsilon_n} = o(\varepsilon_n^{\frac{1}{2}})$,
while $|u^\star_{\varepsilon_n}|_{\varepsilon_n}$, $|u^{\varepsilon_n}|_{\varepsilon_n}$
are $O(\varepsilon_n^{\frac{1}{2}})$, so that the 
functions themselves are larger than their difference. Young's inequality also implies the convergence $|u^{\varepsilon_n} - u^\star_{\varepsilon_n}|_{H^{\frac{1}{2}}(\R^2)} = o(\varepsilon_n^{\frac{1}{2}})$. 
\end{remark}

\bibliographystyle{nonlinearity}
\bibliography{mdg}

\end{document}